\documentclass[sn-mathphys,Numbered]{sn-jnl}
\pdfoutput=1
\usepackage{pgfplots}
\usepackage{graphicx}%
\usepackage{multirow}%
\usepackage{amsmath,amssymb,amsfonts}%
\usepackage{amsthm}%
\usepackage{mathrsfs}%
\usepackage[title]{appendix}%
\usepackage{xcolor}%
\usepackage{textcomp}%
\usepackage{manyfoot}%
\usepackage{booktabs}%
\usepackage{algpseudocode}%
\usepackage{listings}%
\usepackage{ragged2e}
\usepackage{fullpage}
\usepackage{changepage}
\usepackage{stfloats}
\usepackage[caption=false]{subfig}
\usepackage{extarrows}
\usepackage{cleveref}
\usepackage{tikz}
\usepackage{pifont}
\usepackage{enumitem}
\usepackage[ruled, vlined, noline]{algorithm2e}
\usepackage{float}
\usepackage{placeins,needspace} 
\usepackage{graphicx}
\usepackage{subcaption}

\theoremstyle{thmstyleone}%
\newtheorem{theorem}{Theorem}[section]
\newtheorem{proposition}[theorem]{Proposition}%
\theoremstyle{thmstyletwo}%
\newtheorem{remark}{Remark}%
\newtheorem{corollary}{Corollary}%
\theoremstyle{thmstylethree}%
\newtheorem{lemma}{Lemma}[section]
\newtheorem{assumption}[theorem]{Assumption}
\AtBeginDocument{
	\setlength{\abovedisplayskip}{6pt}
	\setlength{\belowdisplayskip}{6pt}
	\setlength{\abovedisplayshortskip}{3pt}
	\setlength{\belowdisplayshortskip}{3pt}
	
	\setlength{\jot}{2pt}
}
\begin{document}
	\title[Tikhonov-Regularized Primal--Dual Methods]{Tikhonov-Regularized Dynamics and Inexact Algorithms for Nonlinear Saddle-Point Problems}
	
	\author[1]{\fnm{Hong-lu} \sur{Li}}\email{lihonglu$\_$7988@163.com}
	
	\author*[1]{\fnm{Yi-bin} \sur{Xiao}}\email{xiaoyb9999@hotmail.com}
	
	
	
	\affil[1]{\orgdiv{School of Mathematical Sciences}, \orgname{University of Electronic Science and Technology of China}, \orgaddress{\street{}\city{Chengdu}, \postcode{611731}, \state{Sichuan}, \country{P.R. China}}}
	
%
	
	\maketitle

\begin{abstract}
	
	\hspace{2em}	\textbf{Abstract.} In this paper, we study a Tikhonov-regularized second-order primal--dual dynamical system for smooth nonlinear	convex--concave saddle-point problems in real Hilbert spaces.	The dynamical system incorporates a vanishing damping coefficient $\alpha/t$, a Hessian-driven term, and a critical Tikhonov regularization coefficient $c/t^2$. By constructing suitable Lyapunov functions, we establish strong convergence to the minimum-norm saddle point and fast convergence rates under the same parameter conditions, without imposing any eventual inside/outside-ball condition. Specifically, the primal--dual trajectory converges strongly to the minimum-norm saddle point, while the primal--dual gap and the velocity norm decay at rates $o(t^{-2})$ and $o(t^{-1})$, respectively. Furthermore, we propose an inexact primal--dual algorithm with a computable relative residual criterion and prove strong convergence of the generated sequence to the same minimum-norm saddle point, together with an $o(k^{-2})$ convergence rate for the primal--dual gap. Finally, numerical experiments on nonlinear and bilinear problems of different dimensions illustrate the minimum-norm selection property and compare the performance of the proposed dynamical system and algorithm with related methods.
	
	\textbf{Keywords.} Nonlinear convex–concave saddle-point problems,  minimum-norm saddle point,  fast convergence rates,  inexact primal–dual algorithm
\end{abstract}

\section{Introduction}\label{sec1}

Let $\mathcal X$ and $\mathcal Y$ be real Hilbert spaces with inner
products $\langle\cdot,\cdot\rangle_{\mathcal X}$ and
$\langle\cdot,\cdot\rangle_{\mathcal Y}$, respectively, and induced
norms $\|\cdot\|_{\mathcal X}$ and $\|\cdot\|_{\mathcal Y}$.
We omit the subscripts when there is no ambiguity and equip
$\mathcal X\times\mathcal Y$ with its standard product inner product
and norm. In this paper, we consider the nonlinear convex--concave
saddle-point problem
\begin{equation}\label{eq:problem}
 \min_{x\in\mathcal X}\max_{y\in\mathcal Y}\mathcal L(x,y),
 \qquad \mathcal L(x,y):=f(x)+\Psi(x,y)-g(y),
\end{equation}
where $f:\mathcal X\to\mathbb R$, $g:\mathcal Y\to\mathbb R$, and
$\Psi:\mathcal X\times\mathcal Y\to\mathbb R$. The coupling function
$\Psi$ is not restricted to be bilinear. Convex--concave saddle-point
formulations arise in matrix games and eigenvalue minimization
\cite{Nemirovski2004}, kernel matrix learning and regression with
fairness constraints \cite{HamedaniAybat2021}, and robust bond
portfolio optimization \cite{LuxenbergSchieleBoyd2024}.
Nonlinear coupling also appears in Lagrangian formulations of convex
optimization problems with nonlinear functional constraints
\cite{HamedaniAybat2021}. These applications motivate the study of primal--dual methods beyond
bilinear coupling. In the present work, we consider such problems
under the full-space smoothness assumptions stated below.

The saddle set of \eqref{eq:problem} is defined by
\begin{equation*}
 \mathcal S:=\{(\bar x,\bar y)\in\mathcal X\times\mathcal Y:
 \mathcal L(\bar x,y)\leq\mathcal L(\bar x,\bar y)
 \leq\mathcal L(x,\bar y)\text{ for all }(x,y)\in\mathcal X\times\mathcal Y\}.
\end{equation*}
Throughout this paper, we impose the following basic assumption.

\begin{assumption}\label{ass:main}
The functions $f:\mathcal X\to\mathbb R$ and
$g:\mathcal Y\to\mathbb R$ are convex and continuously
Fr\'echet differentiable, and
$\Psi:\mathcal X\times\mathcal Y\to\mathbb R$ is continuously
Fr\'echet differentiable. For every $y\in\mathcal Y$, the function
$x\mapsto f(x)+\Psi(x,y)$ is convex; for every $x\in\mathcal X$,
the function $y\mapsto g(y)-\Psi(x,y)$ is convex.
The mappings $(x,y)\mapsto\nabla f(x)+\nabla_x\Psi(x,y)$ and
$(x,y)\mapsto\nabla g(y)-\nabla_y\Psi(x,y)$ are locally Lipschitz
continuous on $\mathcal X\times\mathcal Y$.
The saddle set $\mathcal S$ is nonempty.
\end{assumption}

Under Assumption~\ref{ass:main}, $\mathcal S$ is a nonempty closed
convex set. Hence, its minimum-norm element exists and is unique,
which we denote by
\begin{equation*}
 (x^*,y^*):=P_{\mathcal S}(0,0)
 =\operatorname*{argmin}_{(x,y)\in\mathcal S}
      \bigl\{\|x\|^2+\|y\|^2\bigr\},
\end{equation*}
where $P_{\mathcal S}$ denotes the metric projection onto
$\mathcal S$ in the product Hilbert space.
Throughout, the primal--dual gap is evaluated at this fixed saddle point:
$\mathcal L(x,y^*)-\mathcal L(x^*,y)$.
This quantity differs from restricted gaps obtained by taking a supremum
over test points. To complement this fixed-saddle-point measure,
we also estimate the stationarity residual.

\subsection{Motivation and related works}

First-order primal--dual methods have been widely studied for
solving convex--concave saddle-point problems. For bilinear
coupling, Chambolle et al.
\cite{ChambollePock2011,ChambollePock2016} developed primal--dual
algorithms and established ergodic convergence rates.
However, applications with nonlinear coupling motivate methods
beyond the bilinear setting.
For smooth problems of this type, Nemirovski \cite{Nemirovski2004}
developed the Mirror-Prox method and established an $O(k^{-1})$
primal--dual gap estimate for weighted averaged iterates under
Lipschitz continuity of the saddle operator on convex compact
feasible sets. Incorporating momentum and adaptive stepsizes, Hamedani et al.
\cite{HamedaniAybat2021} introduced a momentum term based on the
partial gradients of the coupling function. They proved
convergence of the iterates and an $O(k^{-1})$ ergodic gap estimate,
and proposed a backtracking strategy that does not require the
Lipschitz constants. When the primal objective is strongly convex
and the coupling is linear in the dual variable, their method
achieves an $O(k^{-2})$ ergodic rate.
To accommodate local smoothness assumptions, Chang et al.
\cite{ChangYangZhang2026}
combined convex combinations with linesearch for general
nonbilinear coupling, proving
convergence of the iterates and an $O(k^{-1})$ ergodic estimate
for a restricted merit function. As an alternative to linesearch,
Huang et al.
\cite{HuangEtAl2026} developed an adaptive primal--dual method
without linesearch, with iterate convergence and an $O(k^{-1})$
ergodic estimate. These ergodic estimates concern averaged iterates.
To obtain last-iterate rates,
Bo\c{t} et al. \cite{BotNguyenZong2025}
proposed a fast reflected forward--backward method for monotone
inclusions and applied it to nonlinear saddle-point problems.
Their analysis establishes weak convergence and last-iterate
$o(k^{-1})$ estimates for the tangent residual and the primal--dual gap.
However, these convergence and gap estimates do not ensure
minimum-norm selection. Combining fast rates with strong convergence
to the minimum-norm saddle point therefore requires further study.

To study convergence rates together with the behavior of the iterates,
it is useful to consider the continuous-time dynamics underlying
accelerated methods. This approach also provides a basis for deriving
discrete algorithms.
For unconstrained convex minimization, Su et al. \cite{SuJMLR2016}
related Nesterov's method to the dynamical system
\begin{equation*}
 \ddot u(t)+\frac{\alpha}{t}\dot u(t)+\nabla\phi(u(t))=0,
\end{equation*}
where $\phi$ is a smooth convex function. For $\alpha\geq3$, they
obtained $\phi(u(t))-\min\phi=O(t^{-2})$. Beyond function values,
subsequent studies established weak trajectory convergence and the
improved $o(t^{-2})$ estimate for $\alpha>3$ in
\cite{AttouchMP2018,May2017}.
To extend this approach to constrained problems, Bo\c{t} et al.
\cite{BotNguyen2021} studied inertial primal--dual dynamics with
vanishing damping and proved trajectory convergence and fast rates.
Following the same dynamical-systems approach, Bo\c{t} et al.
\cite{BotMP2023} derived an accelerated augmented Lagrangian
method with convergence guarantees for the iterates.
For bilinear saddle-point problems,
He et al. \cite{He AMO 2024} investigated second-order
primal--dual dynamics with damping, scaling and extrapolation.
Further addressing trajectory convergence and sharper rates,
He et al. \cite{He Fang 2026}
established trajectory convergence,
an $o(t^{-2})$ primal--dual gap estimate and an $o(t^{-1})$
velocity estimate, and studied corresponding accelerated
discrete algorithms.

To retain fast gap estimates and iterate convergence after
discretization, Ding et al.
\cite{DingFliegeVuong2025} studied a second-order primal--dual
dynamical system for bilinear coupling and its implicit and explicit discretizations.
Their implicit algorithm achieves a nonergodic $O(k^{-2})$ gap
estimate under suitable inertial parameter rules without strong
convexity; for smooth objectives, they also established
nonergodic estimates and weak convergence for an explicit algorithm.
Beyond bilinear coupling, the monotonicity of the saddle-gradient
mapping allows continuous and discrete methods to be studied
for general smooth convex--concave problems. In this direction,
Bo\c{t} et al.
\cite{BotCsetnekNguyen2025} developed the continuous and discrete
Fast OGDA methods. Starting from a second-order dynamical system,
they derived implicit and explicit algorithms and established
weak convergence and last-iterate residual estimates.
More precisely, their implicit method gives $o(k^{-2})$ residuals with
linearly growing scaling, whereas the explicit method gives $o(k^{-1})$
under global Lipschitz continuity. These results establish fast residual
decay, but do not guarantee convergence to the minimum-norm solution.

Despite fast function-value decay, inertial trajectories may
oscillate, motivating additional damping based on local curvature. For unconstrained optimization, Alvarez et al.
\cite{AlvarezEtAl2002} introduced Hessian-driven damping in an
inertial dynamical system related to continuous Newton methods.
Combining this mechanism with vanishing viscous damping for acceleration,
Attouch et al.
\cite{AttouchJDE2016} studied
\begin{equation*}
 \ddot u(t)+\frac{\alpha}{t}\dot u(t)
 +\beta\nabla^2\phi(u(t))\dot u(t)+\nabla\phi(u(t))=0,
 \qquad \beta>0.
\end{equation*}
They established fast function-value and gradient estimates,
together with weak trajectory convergence for $\alpha>3$.
The Hessian-driven term can be written as
$\beta\frac{\mathrm d}{\mathrm dt}\nabla\phi(u(t))$, yielding
an equivalent first-order formulation without explicit Hessian
evaluation. Extending the study of this damping mechanism to
linearly constrained optimization, He et al.
\cite{HeHessian2025} proposed a mixed-order primal--dual dynamical
system with Hessian-driven damping, proved fast Lagrangian,
objective and feasibility estimates, and studied its discretization.
Along a related line, Li et al. \cite{LiHeXiao2025} considered
an implicit form of Hessian damping in primal--dual dynamics.
Extending this mechanism to bilinear saddle-point problems,
Wang et al.
\cite{WangPeypouquet2026} developed Hessian-driven inertial
primal--dual dynamics and established an $O(t^{-2})$ gap estimate
in the convex--concave case, as well as faster estimates under
strong convexity--concavity. A related derivative mechanism is
also present in the Fast OGDA dynamics discussed above
\cite{BotCsetnekNguyen2025}: the derivative acts on the full
monotone primal--dual mapping and therefore also applies to
nonlinear coupling.

Hessian-driven damping does not itself ensure minimum-norm
selection. To obtain this property, vanishing Tikhonov regularization
has been incorporated into inertial dynamics. For unconstrained convex minimization,
Attouch et al. \cite{AttouchJMAA2018} studied
\begin{equation*}
 \ddot u(t)+\frac{\alpha}{t}\dot u(t)
 +\nabla\phi(u(t))+\epsilon(t)u(t)=0,
\end{equation*}
where $\epsilon(t)>0$ tends to zero. Their results distinguish
fast function-value decay and minimum-norm selection according
to the decay of $\epsilon$. For polynomial regularization,
the order $t^{-2}$ separates the fast and slow decay regimes
considered in their analysis, highlighting the importance of
the regularization decay. This leads to the question of whether fast
rates and minimum-norm selection can be obtained simultaneously.
For linearly constrained optimization,
Li et al. \cite{LiHeXiao2025} combined implicit
Hessian damping with Tikhonov regularization and studied fast
rates and minimum-norm strong convergence of the primal trajectory.
To obtain both properties at the critical decay order, Li et al. \cite{Li Xiao 2026} studied
mixed-order primal--dual dynamics and established minimum-norm
strong convergence and little-$o$ rates for the regularization $c/t^2$.
No eventual inside/outside-ball condition relative to the
minimum-norm primal solution is required.

For saddle-point problems, the selection question extends from
the primal variable to the entire primal--dual pair. For this purpose,
Sun et al. \cite{Sun Optimization 2025} combined Hessian-driven
damping, Tikhonov regularization and time scaling for bilinear
saddle-point problems, obtaining strong convergence to the
minimum-norm saddle point and gap estimates.
For monotone equations, which also cover smooth nonlinear
saddle-point problems, Bo\c{t} et al.
\cite{BotNguyen2026} established minimum-norm strong convergence
and residual estimates for Tikhonov-regularized monotone operator
flows. Incorporating inertial and derivative terms into regularized
monotone dynamics, Csetnek et al. \cite{Csetnek Laszlo JEE 2026} studied
minimum-norm selection for a second-order dynamical system.
Using anchoring for discrete solution selection,
Bo\c{t} et al.
\cite{BotChenchene2026} proposed an extragradient method with
flexible anchoring, establishing strong convergence to the
solution nearest the anchor together with residual estimates.
However, combining selection with fast rates depends on the
admissible parameter regimes. In particular,
the second-order analyses in
\cite{Sun Optimization 2025,Csetnek Laszlo JEE 2026} use damping
$\alpha/t^q$ with $0<q<1$ and do not cover damping $\alpha/t$
combined with critical regularization $c/t^2$.
On the other hand, a special case of the dynamical system studied here
reduces, after time rescaling, to the first-order flow in
\cite{BotNguyen2026}. Thus, the remaining continuous-time question is
to establish selection and rates for the full parameter family and
arbitrary initial velocities.

For these continuous-time properties to guide algorithm design,
a further question is whether minimum-norm selection and fast
convergence can be retained after discretization.
Implicit schemes may require solving regularized saddle
subproblems, which generally have no closed-form solutions
for nonlinear coupling.
Since exact solution at every iteration can be costly,
it is natural to consider inexact solutions and examine
whether these convergence properties are preserved.
For bilinear saddle-point problems,
Rasch et al. \cite{RaschChambolle2020} analyzed errors in
proximal maps and gradients, and retained the corresponding exact
algorithm's rates under suitable error decay conditions.
For nonlinear coupling, Khalafi et al. \cite{KhalafiBoob2023}
also allowed inexact subproblem solutions in an accelerated primal--dual method
in the convex--strongly-concave setting. Along a related line for linearly constrained optimization,
He et al. \cite{He JCAM 2025} studied inexact
accelerated augmented Lagrangian algorithms based on approximate
stationarity of the subproblems. Besides allowing inexact solutions,
one must specify how their accuracy is measured and checked.
Rather than fixing an absolute tolerance sequence in advance,
relative error criteria relate the required accuracy to quantities
computed during the iteration. Such criteria have been studied
in hybrid proximal methods \cite{Solodov Svaiter SVA 1999}.
A remaining issue is whether a computable relative residual
condition can preserve minimum-norm strong convergence and
fast last-iterate gap decay for nonlinear saddle-point problems,
without prescribing an absolute error sequence in advance.

Motivated by the above discussion, we study Tikhonov-regularized
primal--dual dynamics for nonlinear saddle-point problems.
Our aim is to combine strong convergence to the minimum-norm
saddle point with fast rates under the critical regularization
$c/t^2$, without imposing any eventual inside/outside-ball
condition on the trajectory, and to retain these properties
in an inexact discretization with a computable relative
residual criterion.

\Needspace{8\baselineskip}
\subsection{Main contributions}

The main contributions of this paper are summarized as follows.
\begin{enumerate}[label=(\alph*),leftmargin=2em]
\item We study a Tikhonov-regularized second-order primal--dual
dynamical system for nonlinear convex--concave saddle-point
problems in real Hilbert spaces. Under Assumption~\ref{ass:main}
and the parameter conditions
$\alpha>3$, $1/(\alpha-1)<\theta<1/2$ and $c>0$, we establish
global existence and uniqueness for arbitrary initial positions and velocities,
and prove that the trajectory
converges strongly to the minimum-norm saddle point $(x^*,y^*)$.
Under the same conditions, we obtain
\begin{equation*}
 \mathcal L(x(t),y^*)-\mathcal L(x^*,y(t))=o(t^{-2}),\quad
 \|\dot x(t)\|+\|\dot y(t)\|=o(t^{-1}),\quad
 \mathcal R(x(t),y(t))=O(t^{-2}),
\end{equation*}
where $\mathcal R(x,y):=(\|\nabla f(x)+\nabla_x\Psi(x,y)\|^2
+\|\nabla g(y)-\nabla_y\Psi(x,y)\|^2)^{1/2}$
is the stationarity residual norm.
These conclusions hold for the critical regularization $c/t^2$
without any eventual inside/outside-ball condition on the
primal--dual trajectory.

\item We propose an inexact primal--dual algorithm obtained by
discretizing an equivalent first-order formulation of the
dynamical system. The regularized saddle subproblems need only
satisfy a computable relative residual criterion. For any fixed
stepsize $s>0$ and a sufficiently small relative error parameter,
we prove strong convergence of $(x_k,y_k)$ to $(x^*,y^*)$ and establish
the following rates in the outer iteration index $k$:
\begin{equation*}
 \mathcal L(x_k,y^*)-\mathcal L(x^*,y_k)=o(k^{-2}),\;
 \|x_k-x_{k-1}\|+\|y_k-y_{k-1}\|=o(k^{-1}),\;
 \mathcal R(x_k,y_k)=O(k^{-2}).
\end{equation*}
The stopping condition requires neither
the exact subproblem solution nor an additional weighted summability
assumption on the errors.

\item We further characterize the asymptotic behavior
of the stationarity residuals. When the minimum-norm
saddle point is nonzero, the residual estimates
$O(t^{-2})$ and $O(k^{-2})$ are sharp along the
generated trajectory and sequence, respectively.
When it is zero, both estimates improve to
little-$o$ rates. Moreover, numerical experiments
on nonlinear and bilinear problems of different
dimensions illustrate minimum-norm selection and
compare the performance of the proposed dynamical
system and algorithm with related methods.
\end{enumerate}

\subsection{Organization}

The rest of this paper is organized as follows. Section~\ref{sec:continuous}
introduces the proposed dynamical system and establishes its
well-posedness, strong convergence and convergence rates.
Section~\ref{sec:discrete} presents the inexact primal--dual algorithm and
analyzes its convergence properties. Section~\ref{sec:numerics} reports
numerical experiments.

\section{Convergence Analysis of the Proposed Dynamical System}\label{sec:continuous}

In this section, we study the following second-order primal--dual dynamical system associated with
\eqref{eq:problem}:
\begin{equation}\label{eq:flow}
	\left\{
	\begin{aligned}
		&\ddot x(t)+\frac{\alpha}{t}\dot x(t)
		+\nabla f(x(t))+\nabla_x\Psi(x(t),y(t))
		+\theta t\frac{\mathrm d}{\mathrm dt}
		\bigl(\nabla f(x(t))+\nabla_x\Psi(x(t),y(t))\bigr)
		+\frac{c}{t^2}x(t)=0,\\
		&\ddot y(t)+\frac{\alpha}{t}\dot y(t)
		+\nabla g(y(t))-\nabla_y\Psi(x(t),y(t))
		+\theta t\frac{\mathrm d}{\mathrm dt}
		\bigl(\nabla g(y(t))-\nabla_y\Psi(x(t),y(t))\bigr)
		+\frac{c}{t^2}y(t)=0,
	\end{aligned}
	\right.
\end{equation}
where $t\geq t_0>0$, the initial positions and velocities are arbitrary, and the parameters satisfy
\begin{equation}\label{eq:parameters}
	\alpha>3,\qquad
	\frac{1}{\alpha-1}<\theta<\frac12,\qquad c>0.
\end{equation}
Under Assumption~\ref{ass:main}, the time derivatives of the gradient terms are understood almost everywhere;
an equivalent first-order formulation will be given below.

The Hessian-driven terms in \eqref{eq:flow} differentiate the full
primal--dual gradients at the current state and thus account for both
objective curvature and nonlinear coupling. This choice differs from
the regularized bilinear dynamics in \cite{Sun Optimization 2025},
where the regularized gradients are evaluated at extrapolated points
before differentiation. To obtain minimum-norm selection in addition to the damping effect,
\eqref{eq:flow} also includes the Tikhonov terms $cx/t^2$ and $cy/t^2$.
Their critical decay is combined with damping $\alpha/t$, while the
coefficient of the saddle-gradient terms remains one. This parameter
choice requires further analysis: for example, the condition
$t^2\beta(t)\epsilon(t)\to+\infty$ in \cite{Zhu JCAM 2025}
does not include $\epsilon(t)=c/t^2$ when $\beta\equiv1$. Therefore, in this section, we investigate whether minimum-norm
selection and fast convergence can be achieved simultaneously for
\eqref{eq:flow}. Under \eqref{eq:parameters}, we first establish global
existence and boundedness, and then prove minimum-norm strong convergence
without an eventual inside/outside-ball condition, together with
the gap, velocity and stationarity residual estimates.

\Needspace{5\baselineskip}
We first give a comparison lemma that will be used to prove
the decay of the energy functions in the convergence analysis.

\begin{lemma}\label{lem:comparison-continuous}
	Let $t_0>0$, $\lambda>0$, and let $H:[t_0,+\infty)\to\mathbb R$ be locally absolutely continuous.
	Assume that there exist $t_1\geq t_0$ and a locally integrable 	function $e:[t_0,+\infty)\to\mathbb R$ with
	$\lim_{t\to\infty}e(t)=0$ such that
	\[
	\dot H(t)+\frac{\lambda}{t}H(t)
	\leq\frac{e(t)}{t}
	\qquad\text{for almost every }t\geq t_1.
	\]
	Then $\limsup_{t\to\infty}H(t)\leq0$.
\end{lemma}
\begin{proof}
	Fix $\varepsilon>0$. Since $e(t)\to0$, there exists	$T\geq t_1$ such that $e(t)\leq\varepsilon$ for all $t\geq T$. Multiplying the differential inequality by $t^\lambda$, we obtain
	\[
	\frac{\mathrm d}{\mathrm dt}\bigl(t^\lambda H(t)\bigr)
	\leq\varepsilon t^{\lambda-1}
	\qquad\text{for almost every }t\geq T.
	\]
	Integrating over $[T,t]$ and dividing by $t^\lambda$ yields
	\[
	H(t)\leq
	\left(\frac{T}{t}\right)^\lambda H(T)
	+\frac{\varepsilon}{\lambda}
	\left[1-\left(\frac{T}{t}\right)^\lambda\right].
	\]
	Thus, $\limsup_{t\to\infty}H(t)\leq\varepsilon/\lambda$. Letting $\varepsilon\downarrow0$ yields the desired conclusion.
	This completes the proof.
\end{proof}

To establish the existence and uniqueness of global solutions
to \eqref{eq:flow}, we first introduce the constants
\begin{equation}\label{eq:constants}
	a=\theta^{-1}-2,\qquad
	r=\alpha+1-\theta^{-1},\qquad
	d=ar-c.
\end{equation}
Under \eqref{eq:parameters}, we have $a>0$, $r>2$, and
\[
a+r=\alpha-1,\qquad \theta(a+2)=1.
\]
With these constants, the following proposition gives an equivalent
first-order formulation and establishes the existence and uniqueness
of global solutions to \eqref{eq:flow}.

\begin{proposition}\label{prop:equivalent}
	Suppose that Assumption~\ref{ass:main} and
	\eqref{eq:parameters} hold, and let $t_0>0$.
	Then, for every initial position
	$(x_0,y_0)\in\mathcal X\times\mathcal Y$ and initial
	velocity $(v_0,w_0)\in\mathcal X\times\mathcal Y$,
	the dynamical system \eqref{eq:flow} admits a unique
	global solution satisfying
	\[
	x(t_0)=x_0,\qquad y(t_0)=y_0,\qquad
	\dot x(t_0)=v_0,\qquad \dot y(t_0)=w_0.
	\]
	Moreover,
	$(x,y)\in C^1([t_0,+\infty);\mathcal X\times\mathcal Y)$,
	and $(\dot x,\dot y)$ is locally absolutely continuous.
	
	Define
	\begin{equation}\label{eq:q}
		\begin{aligned}
			q_x(t)
			&=t\dot x(t)
			+\theta t^2\bigl(\nabla f(x(t))
			+\nabla_x\Psi(x(t),y(t))\bigr)+rx(t),\\
			q_y(t)
			&=t\dot y(t)
			+\theta t^2\bigl(\nabla g(y(t))
			-\nabla_y\Psi(x(t),y(t))\bigr)+ry(t).
		\end{aligned}
	\end{equation}
	Then \eqref{eq:flow} is equivalent to the first-order
	dynamical system
	\begin{equation}\label{eq:firstorder}
		\left\{
		\begin{aligned}
			\dot x
			&=-\theta t\bigl(\nabla f(x)
			+\nabla_x\Psi(x,y)\bigr)
			-\frac{rx-q_x}{t},\\
			\dot y
			&=-\theta t\bigl(\nabla g(y)
			-\nabla_y\Psi(x,y)\bigr)
			-\frac{ry-q_y}{t},\\
			\dot q_x&=-\frac{aq_x-dx}{t},\\
			\dot q_y&=-\frac{aq_y-dy}{t},
		\end{aligned}
		\right.
	\end{equation}
	where the initial values of $q_x,q_y$ are determined
	by \eqref{eq:q} at $t=t_0$.
\end{proposition}

\begin{proof}
	For convenience, write
	$\mathcal H=\mathcal X\times\mathcal Y$,
	$z=(x,y)$, $q=(q_x,q_y)$, and
	\[
	F(z)=
	\bigl(\nabla f(x)+\nabla_x\Psi(x,y),
	\nabla g(y)-\nabla_y\Psi(x,y)\bigr).
	\]
	By Assumption~\ref{ass:main}, $F$ is locally Lipschitz
	continuous on $\mathcal H$.
	Moreover, the convexity of
	$\mathcal L(\cdot,y)$ and $-\mathcal L(x,\cdot)$ gives
	\begin{align*}
		\mathcal L(\widehat x,y)
		&\ge \mathcal L(x,y)
		+\langle\nabla_x\mathcal L(x,y),
		\widehat x-x\rangle,\\
		\mathcal L(x,\widehat y)
		&\le \mathcal L(x,y)
		+\langle\nabla_y\mathcal L(x,y),
		\widehat y-y\rangle.
	\end{align*}
	Consequently,
	\[
	\mathcal L(x,\widehat y)-\mathcal L(\widehat x,y)
	\le \langle F(z),z-\widehat z\rangle,
	\qquad \widehat z=(\widehat x,\widehat y).
	\]
	Interchanging $z$ and $\widehat z$ and adding the
	resulting inequalities, we obtain
	\begin{equation}\label{eq:monotonicity}
		\langle F(z)-F(\widehat z),z-\widehat z\rangle
		\ge0,\qquad z,\widehat z\in\mathcal H.
	\end{equation}
	
	In this notation, \eqref{eq:flow} and \eqref{eq:q}
	become
	\[
	\ddot z+\frac{\alpha}{t}\dot z
	+F(z)+\theta t\frac{\mathrm d}{\mathrm dt}F(z)
	+\frac{c}{t^2}z=0,
	\qquad
	q=t\dot z+\theta t^2F(z)+rz.
	\]
	Along any $C^1$ trajectory, the local Lipschitz
	continuity of $F$ implies that $F(z(\cdot))$ is
	locally Lipschitz in time.
	Hence, for a solution of \eqref{eq:flow},
	differentiation of the expression for $q$ gives,
	almost everywhere,
	\begin{align*}
		\dot q
		&=t\ddot z+(1+r)\dot z+2\theta tF(z)
		+\theta t^2\frac{\mathrm d}{\mathrm dt}F(z)\\
		&=(1+r-\alpha)\dot z+(2\theta-1)tF(z)
		-\frac ct z\\
		&=-a\dot z-a\theta tF(z)-\frac ct z\\
		&=-\frac at q+\frac dt z.
	\end{align*}
	Thus \eqref{eq:firstorder} can be written as
	\begin{equation}\label{eq:firstorder-compact}
		\dot z=-\theta tF(z)-\frac{rz-q}{t},
		\qquad
		\dot q=-\frac{aq-dz}{t}.
	\end{equation}
	Since the right-hand sides are continuous along
	the trajectory, these equations hold everywhere.
	
	Conversely, let $(z,q)$ be a $C^1$ solution of
	\eqref{eq:firstorder-compact}.
	The first equation implies
	$q=t\dot z+\theta t^2F(z)+rz$.
	Since $F(z(\cdot))$ is locally Lipschitz in time,
	the same equation shows that $\dot z$ is locally
	absolutely continuous.
	Differentiating the expression for $q$ and using
	the second equation, we obtain
	\[
	t\ddot z+(1+r+a)\dot z
	+(2+a)\theta tF(z)
	+\theta t^2\frac{\mathrm d}{\mathrm dt}F(z)
	+\frac{ar-d}{t}z=0
	\]
	almost everywhere.
	The identities
	$1+r+a=\alpha$, $(2+a)\theta=1$, and $ar-d=c$
	then recover \eqref{eq:flow}.
	The initial velocities are recovered from
	\eqref{eq:q} at $t=t_0$.
	This proves the equivalence.
	
	The right-hand side of
	\eqref{eq:firstorder-compact} is continuous in time
	and locally Lipschitz in $(z,q)$, locally uniformly
	for $t>0$.
	Therefore, the Cauchy--Lipschitz theorem yields a
	unique maximal $C^1$ solution on
	$[t_0,T_{\max})$ for the prescribed initial data.
	It remains to prove that $T_{\max}=+\infty$.
	
	Suppose, to the contrary, that $T_{\max}<+\infty$.
	Set
	\[
	E(t)=\frac12\bigl(\|z(t)\|^2+\|q(t)\|^2\bigr).
	\]
	By \eqref{eq:firstorder-compact} and
	\eqref{eq:monotonicity}, we have
	\begin{align*}
		E'(t)
		&=-\theta t\langle F(z),z\rangle
		-\frac rt\|z\|^2-\frac at\|q\|^2
		+\frac{1+d}{t}\langle z,q\rangle\\
		&\le \theta t\|F(0)\|\|z\|
		+\frac{|1+d|}{t}\|z\|\|q\|\\
		&\le C_1\bigl(1+E(t)\bigr),
		\qquad t_0\le t<T_{\max},
	\end{align*}
	where $C_1>0$ is independent of $t$.
	Here we used
	$\langle F(z),z\rangle\ge\langle F(0),z\rangle$
	and $a,r>0$.
	Gronwall's inequality implies
	\begin{equation}\label{eq:finite-state-bound}
		\sup_{t_0\le t<T_{\max}}
		\bigl(\|z(t)\|+\|q(t)\|\bigr)<+\infty.
	\end{equation}
	The second equation of
	\eqref{eq:firstorder-compact} also gives
	\[
	\sup_{t_0\le t<T_{\max}}\|\dot q(t)\|<+\infty.
	\]
	
	To estimate $\dot z$, we first observe that
	\eqref{eq:monotonicity} yields
	\[
	\left\langle
	\frac{F(z(t+h))-F(z(t))}{h},
	\frac{z(t+h)-z(t)}{h}
	\right\rangle\ge0.
	\]
	Letting $h\to0$ at a time where $F(z(\cdot))$
	is differentiable, we obtain
	\begin{equation}\label{eq:derivative-positive}
		\left\langle
		\frac{\mathrm d}{\mathrm dt}F(z(t)),\dot z(t)
		\right\rangle\ge0
		\quad\text{for almost every }t.
	\end{equation}
	Differentiating the first equation of
	\eqref{eq:firstorder-compact} and using both
	equations to eliminate $F(z)$ and $\dot q$ gives
	\[
	\ddot z+\theta t\frac{\mathrm d}{\mathrm dt}F(z)
	=\frac{1-r}{t}\dot z
	+\frac{(2r+d)z-(2+a)q}{t^2}
	\]
	almost everywhere.
	Taking the inner product with $\dot z$, and using
	\eqref{eq:finite-state-bound},
	\eqref{eq:derivative-positive}, and $r>2$, we deduce
	\begin{align*}
		\frac12\frac{\mathrm d}{\mathrm dt}\|\dot z\|^2
		&\le \frac{1-r}{t}\|\dot z\|^2
		+\frac{1}{t^2}
		\langle(2r+d)z-(2+a)q,\dot z\rangle\\
		&\le C_2\bigl(1+\|\dot z\|^2\bigr)
	\end{align*}
	for some constant $C_2>0$ independent of $t$.
	Another application of Gronwall's inequality gives
	\[
	\sup_{t_0\le t<T_{\max}}\|\dot z(t)\|<+\infty.
	\]
	
	Consequently, both $z$ and $q$ are Lipschitz
	continuous on $[t_0,T_{\max})$.
	Since $\mathcal H$ is complete, the strong limits
	\[
	z_T=\lim_{t\uparrow T_{\max}}z(t),
	\qquad
	q_T=\lim_{t\uparrow T_{\max}}q(t)
	\]
	exist.
	Applying the local existence theorem at
	$(T_{\max},z_T,q_T)$ extends the solution of
	\eqref{eq:firstorder-compact} beyond $T_{\max}$,
	contradicting its maximality.
	Hence $T_{\max}=+\infty$.
	The equivalence established above proves the
	existence, uniqueness, and stated regularity
	of the global solution to \eqref{eq:flow}.
	This completes the proof.
\end{proof}

The first-order formulation also clarifies the relation to regularized
monotone flows. With $z=(x,y)$, $q=(q_x,q_y)$ and $F$ as above,
if $c=ar$ and $q(t_0)=0$, then $q\equiv0$.
Setting $\tau=\theta t^2/2$ and $u(\tau)=z(t)$ gives
\[
 \frac{\mathrm du}{\mathrm d\tau}+F(u)+\frac{r}{2\tau}u=0.
\]
Since $r/2>1$, \cite[Corollary 2.9]{BotNguyen2026} yields minimum-norm
strong convergence and $\|F(u(\tau))\|=O(\tau^{-1})$ in this case.
Thus its $O(t^{-2})$ residual bound also follows by time rescaling.
For general $c>0$ and arbitrary initial velocities, the auxiliary variable
remains coupled to the primal--dual trajectory. We therefore use a
quadratic energy to control both variables for each sign of $d$.
The saddle structure will then allow us to derive the gap and
fixed-variable estimates.

\Needspace{5\baselineskip}
To prepare for the strong convergence analysis, the following lemma
establishes boundedness and preliminary velocity and residual estimates.

\begin{lemma}\label{lem:bounded}
	Suppose that Assumption~\ref{ass:main} and
	\eqref{eq:parameters} hold.
	Let $(x,y)$ be the global solution of \eqref{eq:flow}
	given by Proposition~\ref{prop:equivalent}, and let
	$q_x,q_y$ be defined by \eqref{eq:q}.
	Then there exists a constant $C>0$ such that,
	for all $t\ge t_0$,
	\begin{equation}\label{eq:bounded}
		\begin{aligned}
			&\|x(t)\|+\|y(t)\|
			+\|q_x(t)\|+\|q_y(t)\|\le C,\\
			&\|\dot x(t)\|^2+\|\dot y(t)\|^2
			\le \frac{C}{t^2},\qquad
			\mathcal R(x(t),y(t))\le\frac{C}{t^2},
		\end{aligned}
	\end{equation}
	where
	\[
	\mathcal R(x,y):=
	\left(
	\|\nabla f(x)+\nabla_x\Psi(x,y)\|^2
	+\|\nabla g(y)-\nabla_y\Psi(x,y)\|^2
	\right)^{1/2}.
	\]
\end{lemma}

\begin{proof}
	Write $\mathcal H=\mathcal X\times\mathcal Y$,
	$z=(x,y)$, $z^*=(x^*,y^*)$, $q=(q_x,q_y)$, and
	\[
	F(z)=
	\bigl(\nabla f(x)+\nabla_x\Psi(x,y),
	\nabla g(y)-\nabla_y\Psi(x,y)\bigr),
	\]
	where $z^*$ is the minimum-norm saddle point.
	The first-order optimality conditions at $z^*$
	give $F(z^*)=0$.
	Hence, by \eqref{eq:monotonicity},
	\[
	\langle F(z),z-z^*\rangle\ge0.
	\]
	
	Choose
	\begin{equation}\label{eq:weight}
		\varpi=
		\begin{cases}
			|d|^{-1},& d\ne0,\\
			(ra)^{-1},& d=0,
		\end{cases}
	\end{equation}
	and define
	\begin{equation}\label{eq:energy}
		\mathcal E(t)
		=\frac12\|z(t)-z^*\|^2
		+\frac{\varpi}{2}
		\left\|q(t)-\frac daz^*\right\|^2.
	\end{equation}
	Since $r-d/a=c/a$, the first-order equations yield
	\[
	\dot z
	=-\theta tF(z)-\frac rt(z-z^*)
	+\frac1t\left(q-\frac daz^*\right)
	-\frac{c}{at}z^*,
	\qquad
	\dot q
	=-\frac at\left(q-\frac daz^*\right)
	+\frac dt(z-z^*).
	\]
	Consequently,
	\begin{align}\label{eq:energyidentity}
		\mathcal E'(t)
		&=-\theta t\langle F(z),z-z^*\rangle
		-\frac{c}{at}\langle z^*,z-z^*\rangle
		\notag\\
		&\quad-\frac1t\left\{
		r\|z-z^*\|^2
		+\varpi a\left\|q-\frac daz^*\right\|^2
		-(1+\varpi d)
		\left\langle z-z^*,q-\frac daz^*\right\rangle
		\right\}.
	\end{align}
	
	To estimate the quadratic expression, consider
	\[
	Q(v,w)
	=r\|v\|^2+\varpi a\|w\|^2
	-(1+\varpi d)\langle v,w\rangle,
	\qquad v,w\in\mathcal H.
	\]
	In the variables $(v,\sqrt{\varpi}w)$, its
	coefficient matrix is
	\[
	M=
	\begin{cases}
		\begin{pmatrix}
			r&-\sqrt d\\
			-\sqrt d&a
		\end{pmatrix},&d>0,\\[6pt]
		\begin{pmatrix}
			r&0\\
			0&a
		\end{pmatrix},&d<0,\\[6pt]
		\begin{pmatrix}
			r&-\sqrt{ra}/2\\
			-\sqrt{ra}/2&a
		\end{pmatrix},&d=0.
	\end{cases}
	\]
	Its diagonal entries are positive, and its
	determinant equals $c$, $ra$, and $3ra/4$,
	respectively.
	Thus $M$ is positive definite.
	Taking $\mu>0$ to be its smallest eigenvalue
	and applying an orthogonal diagonalization gives
	\begin{equation}\label{eq:coercivity}
		r\|v\|^2+\varpi a\|w\|^2
		-(1+\varpi d)\langle v,w\rangle
		\ge\mu\bigl(\|v\|^2+\varpi\|w\|^2\bigr),
		\qquad v,w\in\mathcal H.
	\end{equation}
	Applying this inequality with
	$v=z-z^*$ and $w=q-(d/a)z^*$ in
	\eqref{eq:energyidentity}, we obtain
	\begin{equation}\label{eq:energyineq}
		\mathcal E'(t)+\frac{2\mu}{t}\mathcal E(t)
		\le-\frac{c}{at}\langle z^*,z(t)-z^*\rangle.
	\end{equation}
	Moreover, Young's inequality gives
	\[
	-\frac ca\langle z^*,z-z^*\rangle
	\le\frac{\mu}{2}\|z-z^*\|^2
	+\frac{c^2\|z^*\|^2}{2\mu a^2}
	\le\mu\mathcal E(t)
	+\frac{c^2\|z^*\|^2}{2\mu a^2}.
	\]
	Therefore,
	\[
	\mathcal E'(t)+\frac{\mu}{t}\mathcal E(t)
	\le\frac{c^2\|z^*\|^2}{2\mu a^2t}.
	\]
	Multiplying by $t^\mu$ and integrating, we deduce
	\[
	\mathcal E(t)
	\le\left(\frac{t_0}{t}\right)^\mu\mathcal E(t_0)
	+\frac{c^2\|z^*\|^2}{2\mu^2a^2}
	\left\{1-\left(\frac{t_0}{t}\right)^\mu\right\}.
	\]
	Hence $z$ and $q$ are bounded on $[t_0,+\infty)$.
	
	We next estimate the velocity.
	As established in the proof of
	Proposition~\ref{prop:equivalent},
	\begin{equation}\label{eq:velocity}
		\ddot z+\theta t\frac{\mathrm d}{\mathrm dt}F(z)
		+\frac{r-1}{t}\dot z
		=\frac{(2r+d)z-(a+2)q}{t^2}
	\end{equation}
	almost everywhere.
	Taking the inner product with $\dot z$ and using
	\eqref{eq:derivative-positive} and the boundedness
	of $z,q$, we obtain
	\[
	\frac12\frac{\mathrm d}{\mathrm dt}\|\dot z\|^2
	+\frac{r-1}{t}\|\dot z\|^2
	\le\frac{C_0}{t^2}\|\dot z\|
	\]
	for some constant $C_0>0$.
	Since $r>2$, Young's inequality yields
	\[
	\frac{C_0}{t^2}\|\dot z\|
	\le\frac{r-2}{2t}\|\dot z\|^2
	+\frac{C_0^2}{2(r-2)t^3}.
	\]
	Setting $u(t)=\|\dot z(t)\|^2$, we thus have
	\[
	u'(t)+\frac rt u(t)\le\frac{C_1}{t^3},
	\qquad C_1=\frac{C_0^2}{r-2},
	\]
	almost everywhere.
	Multiplication by $t^r$ and integration give
	\[
	u(t)
	\le\left(\frac{t_0}{t}\right)^r u(t_0)
	+\frac{C_1}{r-2}
	\left(\frac1{t^2}-\frac{t_0^{r-2}}{t^r}\right).
	\]
	It follows that $\|\dot z(t)\|^2\le C_2/t^2$
	for some constant $C_2>0$.
	
	Finally, \eqref{eq:q} implies
	\begin{equation}\label{eq:residualidentity}
		\theta t^2F(z(t))
		=q(t)-rz(t)-t\dot z(t).
	\end{equation}
	The right-hand side is bounded on $[t_0,+\infty)$.
	Since $\mathcal R(x(t),y(t))=\|F(z(t))\|$,
	the residual estimate follows.
	Enlarging the constant if necessary proves
	\eqref{eq:bounded}.
	This completes the proof.
\end{proof}

\Needspace{5\baselineskip}
We now use these estimates to prove minimum-norm strong convergence
and the little-$o$ rates for the primal--dual gap and velocity.

\begin{theorem}\label{thm:continuous}
	Suppose that Assumption~\ref{ass:main} and
	\eqref{eq:parameters} hold.
	Let $(x,y)$ be the solution of \eqref{eq:flow}.
	Then the trajectory converges strongly to the
	minimum-norm saddle point $(x^*,y^*)$, namely,
	\begin{equation}\label{eq:strong}
		x(t)\longrightarrow x^*,\qquad
		y(t)\longrightarrow y^*
		\quad\text{as }t\to+\infty.
	\end{equation}
	Moreover, as $t\to+\infty$,
\begin{equation}\label{eq:rates}
	\|\dot x(t)\|+\|\dot y(t)\|=o(t^{-1}),
	\quad
	\mathcal R(x(t),y(t))=O(t^{-2})\quad
	\text{and}\quad
	\mathcal L(x(t),y^*)-\mathcal L(x^*,y(t))
	=o(t^{-2}),
\end{equation}
	where $\mathcal R$ is the stationarity residual
	defined in Lemma~\ref{lem:bounded}.
\end{theorem}

\begin{proof}
	Write $\mathcal H=\mathcal X\times\mathcal Y$,
	$z=(x,y)$, $z^*=(x^*,y^*)$, $q=(q_x,q_y)$, and
	\[
	F(z)=
	\bigl(\nabla f(x)+\nabla_x\Psi(x,y),
	\nabla g(y)-\nabla_y\Psi(x,y)\bigr).
	\]
	By Lemma~\ref{lem:bounded}, $z,q$, and $t\dot z$
	are bounded, and $\|F(z(t))\|\to0$.
	
	We first show that every weak sequential cluster
	point of $z(t)$ belongs to $\mathcal S$.
	Let $t_j\to+\infty$ satisfy
	$z(t_j)\rightharpoonup\bar z$.
	By \eqref{eq:monotonicity}, for every
	$v\in\mathcal H$,
	\[
	\langle F(v),v-z(t_j)\rangle
	\ge\langle F(z(t_j)),v-z(t_j)\rangle.
	\]
	Since $(z(t_j))$ is bounded and
	$F(z(t_j))\to0$, passing to the limit gives
	\[
	\langle F(v),v-\bar z\rangle\ge0,
	\qquad v\in\mathcal H.
	\]
	Taking $v=\bar z-sF(\bar z)$ with $s>0$ yields
	\[
	\langle F(\bar z-sF(\bar z)),F(\bar z)\rangle
	\le0.
	\]
	Letting $s\downarrow0$ and using the continuity
	of $F$, we obtain $\|F(\bar z)\|^2\le0$.
	Thus $F(\bar z)=0$, and the first-order
	optimality conditions for the convex--concave
	function $\mathcal L$ imply $\bar z\in\mathcal S$.
	
	We next use this cluster-point property to prove minimum-norm
	selection. Since $z^*$ is the projection of the origin onto
	the closed convex set $\mathcal S$, we have
	\begin{equation}\label{eq:projection}
		\langle z^*,v-z^*\rangle\ge0,
		\qquad v\in\mathcal S.
	\end{equation}
	We claim that
	\begin{equation}\label{eq:selection-liminf}
		\liminf_{t\to+\infty}
		\langle z^*,z(t)-z^*\rangle\ge0.
	\end{equation}
	Indeed, the boundedness of $z$ allows us to choose
	a sequence $t_j\to+\infty$ along which the scalar
	function $\langle z^*,z(t)-z^*\rangle$ converges
	to its lower limit.
	Passing to a subsequence, we may assume that
	$z(t_j)\rightharpoonup\bar z$.
	The preceding argument gives $\bar z\in\mathcal S$,
	and hence
	\[
	\liminf_{t\to+\infty}
	\langle z^*,z(t)-z^*\rangle
	=\langle z^*,\bar z-z^*\rangle\ge0.
	\]
	
	This lower-limit estimate allows us to control the remaining term
	in the energy inequality. Let $\mathcal E$ be defined by
	\eqref{eq:energy}. It follows from \eqref{eq:energyineq} that
	\[
	\mathcal E'(t)+\frac{2\mu}{t}\mathcal E(t)
	\le\frac{e_1(t)}{t},
	\qquad
	e_1(t)=\frac ca
	\max\{0,-\langle z^*,z(t)-z^*\rangle\}.
	\]
	By \eqref{eq:selection-liminf}, $e_1(t)\to0$.
	Applying Lemma~\ref{lem:comparison-continuous}
	and using $\mathcal E(t)\ge0$, we conclude that
	$\mathcal E(t)\to0$.
	Therefore, \eqref{eq:strong} holds, and
	\begin{equation}\label{eq:auxiliarylimits}
		q_x(t)\longrightarrow\frac da x^*,
		\qquad
		q_y(t)\longrightarrow\frac da y^*.
	\end{equation}
	
	To improve the velocity estimate, define
	\begin{equation}\label{eq:secondenergy}
		\mathcal V(t)
		=\frac12t^2\|\dot z(t)\|^2
		-\frac ca\bigl(\|z(t)\|^2-\|z^*\|^2\bigr).
	\end{equation}
	Using $ar-d=c$, we have
	\[
	(2r+d)z-(a+2)q
	=\frac{2c}{a}z-(a+2)
	\left(q-\frac daz\right).
	\]
	Differentiating \eqref{eq:secondenergy} and
	applying \eqref{eq:velocity}, we obtain
	\begin{align*}
		\mathcal V'(t)
		&=t\|\dot z\|^2+t^2\langle\ddot z,\dot z\rangle
		-\frac{2c}{a}\langle z,\dot z\rangle\\
		&=-\theta t^3
		\left\langle
		\frac{\mathrm d}{\mathrm dt}F(z),\dot z
		\right\rangle
		-(r-2)t\|\dot z\|^2
		-(a+2)
		\left\langle q-\frac daz,\dot z\right\rangle
	\end{align*}
	almost everywhere.
	Consequently, \eqref{eq:derivative-positive} gives
	\begin{equation}\label{eq:secondineq}
		t\mathcal V'(t)+2(r-2)\mathcal V(t)
		\le e_2(t),
	\end{equation}
	where
	\[
	e_2(t)
	=-\frac{2(r-2)c}{a}
	\bigl(\|z(t)\|^2-\|z^*\|^2\bigr)
	-(a+2)
	\left\langle
	q(t)-\frac daz(t),t\dot z(t)
	\right\rangle.
	\]
	By \eqref{eq:strong} and \eqref{eq:auxiliarylimits},
	\[
	\|z(t)\|^2-\|z^*\|^2\to0,
	\qquad
	q(t)-\frac daz(t)\to0.
	\]
	Since $t\dot z(t)$ is bounded,
	it follows that $e_2(t)\to0$.
	Applying Lemma~\ref{lem:comparison-continuous}
	to \eqref{eq:secondineq}, with $2(r-2)>0$,
	yields
	\[
	\limsup_{t\to+\infty}\mathcal V(t)\le0.
	\]
	Moreover, \eqref{eq:strong} implies
	\[
	\mathcal V(t)-\frac12t^2\|\dot z(t)\|^2
	\longrightarrow0.
	\]
	Hence $t^2\|\dot z(t)\|^2\to0$.
	Since
	\[
	\|\dot x(t)\|+\|\dot y(t)\|
	\le\sqrt{2}\|\dot z(t)\|,
	\]
	this proves the velocity estimate in \eqref{eq:rates}.
	The residual estimate in \eqref{eq:rates}
	follows directly from Lemma~\ref{lem:bounded}.
	
	Finally, the saddle inequalities and the convexity
	of $\mathcal L(\cdot,y)$ and $-\mathcal L(x,\cdot)$
	give
	\begin{align*}
		0
		&\le\mathcal L(x(t),y^*)
		-\mathcal L(x^*,y(t))\\
		&\le
		\langle\nabla f(x(t))
		+\nabla_x\Psi(x(t),y(t)),x(t)-x^*\rangle\\
		&\quad+
		\langle\nabla g(y(t))
		-\nabla_y\Psi(x(t),y(t)),y(t)-y^*\rangle\\
		&\le\mathcal R(x(t),y(t))\|z(t)-z^*\|.
	\end{align*}
	Multiplying by $t^2$ and using
	\eqref{eq:bounded} and \eqref{eq:strong}, we obtain
	\[
	0\le t^2\bigl(
	\mathcal L(x(t),y^*)-\mathcal L(x^*,y(t))
	\bigr)
	\le C\|z(t)-z^*\|\longrightarrow0.
	\]
	This proves the primal--dual gap estimate
	in \eqref{eq:rates}.
	This completes the proof.
\end{proof}

\begin{remark}
	For bilinear saddle-point problems, accelerated	primal--dual dynamics yield $O(t^{-2})$ gap	estimates \cite{He AMO 2024}, while trajectory	convergence and the rates $o(t^{-2})$ for the	gap and $o(t^{-1})$ for the velocity are	established in \cite{He Fang 2026}.	Theorem~\ref{thm:continuous} obtains these
	little-$o$ rates for nonlinear convex--concave	saddle-point problems in real Hilbert spaces,	together with strong convergence to the	minimum-norm saddle point and an $O(t^{-2})$	stationarity residual estimate.	Thus, in the bilinear case, the proposed	regularized dynamical system achieves the
	same gap and velocity rates with the additional	minimum-norm selection property. In addition, the combination of damping $\alpha/t$ and	critical regularization $c/t^2$ lies outside	the parameter regimes of the strong convergence	theorems in	\cite{Sun Optimization 2025,Csetnek Laszlo JEE 2026}.
	For finite-dimensional linearly constrained	optimization, \cite{Li Xiao 2026} establishes	strong convergence and little-$o$ rates under	critical regularization without eventual	inside/outside-ball conditions, selecting the	minimum-norm primal solution and a compatible	KKT multiplier.
	Here, minimum-norm selection concerns the	entire primal--dual pair in the nonlinear Hilbert-space setting, with strong convergence	and all three rate estimates obtained under	the same parameter conditions and without	eventual inside/outside-ball assumptions.
\end{remark}

\Needspace{6\baselineskip}
The next corollary identifies the scaled residual limits and shows
when the residual estimate is sharp.

\begin{corollary}\label{cor:separation}
	Under the assumptions of Theorem~\ref{thm:continuous},
	\begin{equation}\label{eq:residuallimit}
		\begin{aligned}
			t^2\bigl(\nabla f(x(t))
			+\nabla_x\Psi(x(t),y(t))\bigr)
			&\longrightarrow-\frac{c}{1-2\theta}x^*,\\
			t^2\bigl(\nabla g(y(t))
			-\nabla_y\Psi(x(t),y(t))\bigr)
			&\longrightarrow-\frac{c}{1-2\theta}y^*
		\end{aligned}
	\end{equation}
	strongly as $t\to+\infty$.
	Consequently,
	\begin{equation}\label{eq:residualnormlimit}
		\lim_{t\to+\infty}t^2\mathcal R(x(t),y(t))
		=\frac{c}{1-2\theta}
		\sqrt{\|x^*\|^2+\|y^*\|^2}.
	\end{equation}
	In particular, if $(x^*,y^*)\ne(0,0)$, the
	$O(t^{-2})$ residual estimate in \eqref{eq:rates}
	cannot be improved to $o(t^{-2})$.
	If $(x^*,y^*)=(0,0)$, then
	$\mathcal R(x(t),y(t))=o(t^{-2})$.
\end{corollary}

\begin{proof}
	Write $z=(x,y)$, $z^*=(x^*,y^*)$, and
	$q=(q_x,q_y)$.
	By \eqref{eq:residualidentity},
	\[
	t^2
	\begin{pmatrix}
		\nabla f(x(t))+\nabla_x\Psi(x(t),y(t))\\
		\nabla g(y(t))-\nabla_y\Psi(x(t),y(t))
	\end{pmatrix}
	=\frac1\theta
	\bigl(q(t)-rz(t)-t\dot z(t)\bigr).
	\]
	Theorem~\ref{thm:continuous} and
	\eqref{eq:auxiliarylimits} give
	\[
	z(t)\to z^*,\qquad
	t\dot z(t)\to0,\qquad
	q(t)\to\frac daz^*.
	\]
	Hence the right-hand side converges strongly to
	\[
	\frac1\theta\left(\frac da-r\right)z^*
	=-\frac{c}{a\theta}z^*
	=-\frac{c}{1-2\theta}z^*,
	\]
	where we used $ar-d=c$ and $a\theta=1-2\theta$.
	This proves \eqref{eq:residuallimit}.
	Taking the norm in $\mathcal X\times\mathcal Y$
	yields \eqref{eq:residualnormlimit}, from which
	the remaining assertions follow.
	This completes the proof.
\end{proof}

Although the stationarity residual has exact order $t^{-2}$
when $(x^*,y^*)\ne(0,0)$, the primal--dual gap decays as
$o(t^{-2})$ by \eqref{eq:rates}.
Indeed, the convexity assumptions give
\[
0\le\mathcal L(x(t),y^*)-\mathcal L(x^*,y(t))
\le\mathcal R(x(t),y(t))
\sqrt{\|x(t)-x^*\|^2+\|y(t)-y^*\|^2},
\]
and the distance factor tends to zero.
The gap estimate also yields the following convergence
rates when one variable is fixed at its limiting value.

\begin{corollary}\label{cor:continuous-slices}
	Under the assumptions of Theorem~\ref{thm:continuous},
	as $t\to+\infty$,
	\begin{equation}\label{eq:slice-values}
		\begin{aligned}
			f(x(t))-f(x^*)
			+\Psi(x(t),y^*)-\Psi(x^*,y^*)
			&=o(t^{-2}),\\
			g(y(t))-g(y^*)
			+\Psi(x^*,y^*)-\Psi(x^*,y(t))
			&=o(t^{-2}).
		\end{aligned}
	\end{equation}
	Moreover,
	\begin{equation}\label{eq:slice-gradients}
		\|\nabla f(x(t))+\nabla_x\Psi(x(t),y^*)\|
		+\|\nabla g(y(t))-\nabla_y\Psi(x^*,y(t))\|
		=o(t^{-1}).
	\end{equation}
\end{corollary}

\begin{proof}
	Define
	\[
	\phi(x)=f(x)+\Psi(x,y^*),
	\qquad
	\psi(y)=g(y)-\Psi(x^*,y).
	\]
	By Assumption~\ref{ass:main} and the saddle
	inequalities, $\phi,\psi$ are convex and attain
	their minima at $x^*,y^*$, respectively.
	Moreover,
	\begin{align*}
		&\phi(x(t))-\phi(x^*)
		+\psi(y(t))-\psi(y^*)\\
		&\qquad=
		\mathcal L(x(t),y^*)-\mathcal L(x^*,y(t))
		=o(t^{-2})
	\end{align*}
	by \eqref{eq:rates}.
	Since both differences on the left are nonnegative,
	\eqref{eq:slice-values} follows.
	
	To estimate the gradients, Assumption~\ref{ass:main}
	provides constants $\delta>0$ and $L>0$ such that
	$\nabla\phi$ and $\nabla\psi$ are $L$-Lipschitz
	on $B_\delta(x^*)$ and $B_\delta(y^*)$, respectively.
	For $x\in B_{\delta/2}(x^*)$, set
	\[
	\widetilde x=x-\frac1L\nabla\phi(x).
	\]
	Since $\nabla\phi(x^*)=0$, we have
	\[
	\|\widetilde x-x^*\|
	\le\|x-x^*\|+\frac1L\|\nabla\phi(x)\|
	\le2\|x-x^*\|<\delta.
	\]
	Thus the segment joining $x$ and $\widetilde x$
	lies in $B_\delta(x^*)$.
	The descent inequality on this ball gives
	\[
	\phi(x^*)\le\phi(\widetilde x)
	\le\phi(x)-\frac1{2L}\|\nabla\phi(x)\|^2,
	\]
	and hence
	\[
	\|\nabla\phi(x)\|^2
	\le2L\bigl(\phi(x)-\phi(x^*)\bigr).
	\]
	Similarly, for $y\in B_{\delta/2}(y^*)$,
	\[
	\|\nabla\psi(y)\|^2
	\le2L\bigl(\psi(y)-\psi(y^*)\bigr).
	\]
	
	By \eqref{eq:strong}, these estimates apply to
	$x(t),y(t)$ for all sufficiently large $t$.
	Therefore,
	\begin{align*}
		&\bigl(\|\nabla\phi(x(t))\|
		+\|\nabla\psi(y(t))\|\bigr)^2
		\le
		2\|\nabla\phi(x(t))\|^2
		+2\|\nabla\psi(y(t))\|^2\\
		&\qquad\le
		4L\bigl(
		\mathcal L(x(t),y^*)-\mathcal L(x^*,y(t))
		\bigr)
		=o(t^{-2}).
	\end{align*}
	Taking square roots proves
	\eqref{eq:slice-gradients}.
	This completes the proof.
\end{proof}

\section{An inexact primal--dual algorithm}\label{sec:discrete}

In this section, we propose an inexact primal--dual
algorithm for \eqref{eq:problem} based on the first-order
formulation \eqref{eq:firstorder}.
At each iteration, a quadratically regularized saddle
subproblem is solved approximately under a relative
error criterion.
We establish strong convergence of the iterates to the
minimum-norm saddle point and an $o(k^{-2})$ convergence
rate for the primal--dual gap.
Throughout this section, Assumption~\ref{ass:main} and
the parameter conditions \eqref{eq:parameters} hold,
and $a,r,d$ are defined by \eqref{eq:constants}.

To specify the updates and their required accuracy, let $s>0$ and
$0<\sigma<(r-2)/2$. The resulting method is summarized in
Algorithm~\ref{alg:relative}.
\begin{algorithm}[htbp]
	\caption{Inexact primal--dual algorithm}
	\label{alg:relative}
	\DontPrintSemicolon
	
	\textbf{Parameters:}
	Choose $s>0$, $\alpha>3$,
	$1/(\alpha-1)<\theta<1/2$, and $c>0$.
	Set $a=\theta^{-1}-2$, $r=\alpha+1-\theta^{-1}$,
	$d=ar-c$, and choose $0<\sigma<(r-2)/2$\;
	
	\textbf{Initialization:}
	Choose $x_1,q_{x,1}\in\mathcal X$ and
	$y_1,q_{y,1}\in\mathcal Y$\;
	
	\For{$k=2,3,\ldots$}{
		\textbf{Step 1:} Compute
		\[
		q_{x,k}=\frac{kq_{x,k-1}+dx_{k-1}}{k+a},
		\qquad
		q_{y,k}=\frac{kq_{y,k-1}+dy_{k-1}}{k+a}.
		\]
		
		\textbf{Step 2:} Set
		$\eta_k=\theta s^2k^2/(k+r)$,
		$v_{x,k}=(kx_{k-1}+q_{x,k})/(k+r)$, and
		$v_{y,k}=(ky_{k-1}+q_{y,k})/(k+r)$\;
		
		\textbf{Step 3:} For each inner candidate
		$(x,y)$, evaluate
		\[
		\begin{aligned}
			R_{x,k}
			=x-v_{x,k}
			+\eta_k\bigl(
			\nabla f(x)+\nabla_x\Psi(x,y)
			\bigr),\quad
			R_{y,k}
			=y-v_{y,k}
			+\eta_k\bigl(
			\nabla g(y)-\nabla_y\Psi(x,y)
			\bigr).
		\end{aligned}
		\]
		First evaluate these residuals at
		$(x,y)=(x_{k-1},y_{k-1})$\;
		
		\eIf{$\|R_{x,k}\|^2+\|R_{y,k}\|^2=0$}{
			Set $(x_k,y_k)=(x_{k-1},y_{k-1})$\;
		}{
			Approximately solve
			\[
			\min_{x\in\mathcal X}\max_{y\in\mathcal Y}
			\left\{
			f(x)+\Psi(x,y)-g(y)
			+\frac{\|x-v_{x,k}\|^2}{2\eta_k}
			-\frac{\|y-v_{y,k}\|^2}{2\eta_k}
			\right\}
			\]
			until the current candidate $(x,y)$
			satisfies
			\[
			\bigl(\|R_{x,k}\|^2+\|R_{y,k}\|^2\bigr)^{1/2}
			\le\frac{\sigma}{k+r}
			\bigl(
			\|x-x_{k-1}\|^2+\|y-y_{k-1}\|^2
			\bigr)^{1/2}.
			\]
			Set $(x_k,y_k)=(x,y)$\;
		}
	}
\end{algorithm}

To explain how Algorithm~\ref{alg:relative} discretizes
\eqref{eq:firstorder}, define the scaled errors at the accepted iterate by
\begin{equation}\label{eq:xi-residual}
	\xi_{x,k}=\frac{k+r}{k}R_{x,k},
	\qquad
	\xi_{y,k}=\frac{k+r}{k}R_{y,k}.
\end{equation}
By the definitions of $\eta_k$, $v_{x,k}$, $v_{y,k}$
and the residuals in Algorithm~\ref{alg:relative},
the iterates satisfy
\begin{equation}\label{eq:scheme}
	\left\{
	\begin{aligned}
		x_k-x_{k-1}
		&=-\theta s^2k
		\bigl(\nabla f(x_k)+\nabla_x\Psi(x_k,y_k)\bigr)
		-\frac{rx_k-q_{x,k}}{k}+\xi_{x,k},\\
		y_k-y_{k-1}
		&=-\theta s^2k
		\bigl(\nabla g(y_k)-\nabla_y\Psi(x_k,y_k)\bigr)
		-\frac{ry_k-q_{y,k}}{k}+\xi_{y,k},\\
		q_{x,k}-q_{x,k-1}
		&=-\frac{aq_{x,k}-dx_{k-1}}{k},\\
		q_{y,k}-q_{y,k-1}
		&=-\frac{aq_{y,k}-dy_{k-1}}{k}.
	\end{aligned}
	\right.
\end{equation}
With $t_k=ks$, the exact case
$\xi_{x,k}=\xi_{y,k}=0$ corresponds to a discretization
of \eqref{eq:firstorder} using backward differences,
with the first two equations evaluated at the new
iterate and the terms $dx,dy$ in the auxiliary
equations evaluated at the preceding iterate.
Thus the auxiliary variables can be updated before
solving the regularized saddle subproblem in Step~3
of Algorithm~\ref{alg:relative}. For convenience, write $z_k=(x_k,y_k)$,
$q_k=(q_{x,k},q_{y,k})$, and
$\xi_k=(\xi_{x,k},\xi_{y,k})$, and set
\[
F(z)=
\bigl(
\nabla f(x)+\nabla_x\Psi(x,y),
\nabla g(y)-\nabla_y\Psi(x,y)
\bigr),
\qquad z=(x,y).
\]
Then \eqref{eq:scheme} and the stopping criterion
in Step~3 of Algorithm~\ref{alg:relative} give
\begin{equation}\label{eq:discrete-compact}
	\left\{
	\begin{aligned}
		z_k-z_{k-1}
		&=-\theta s^2kF(z_k)
		-\frac{rz_k-q_k}{k}+\xi_k,\\
		q_k-q_{k-1}
		&=-\frac{aq_k-dz_{k-1}}{k},
	\end{aligned}
	\right.
	\qquad
	\|\xi_k\|\le\frac{\sigma}{k}\|z_k-z_{k-1}\|.
\end{equation}
At each iteration, the regularized saddle subproblem
is equivalent to
$F(z)+\eta_k^{-1}(z-v_k)=0$, where
$v_k=(v_{x,k},v_{y,k})$.
Since $F$ is continuous and monotone on the whole
space, Minty's theorem \cite{Minty1962} ensures
a unique solution
$\widehat z_k=(I+\eta_kF)^{-1}v_k$.
This subproblem can be solved approximately by
Tseng's forward--backward--forward method
\cite{Tseng SICON 2000}, under its regularity
and stepsize conditions.
For $\sigma>0$, any inner iteration converging
strongly to $\widehat z_k$ satisfies Step~3
after finitely many iterations, provided that
$z_{k-1}$ is checked initially.

The stopping criterion controls the subproblem
error without requiring knowledge of
$\widehat z_k$. Indeed, monotonicity gives
\[
\|z_k-\widehat z_k\|
\le\|R_k\|
\le\frac{\sigma}{k+r}\|z_k-z_{k-1}\|,
\qquad R_k=(R_{x,k},R_{y,k}).
\]
Relative residual criteria have also been studied
in hybrid proximal methods
\cite{Solodov Svaiter SVA 1999}.
Here, the factor $(k+r)^{-1}$ yields the error
bound in \eqref{eq:discrete-compact}, allowing
the perturbations to be controlled in the
energy estimates under $2\sigma<r-2$.
In the inexact augmented Lagrangian method
of \cite{He JCAM 2025}, the prescribed stationarity
tolerances satisfy a weighted summability condition.
Our criterion instead uses the current increment
to determine the required accuracy, without
prescribing an absolute tolerance sequence or
imposing a separate weighted summability condition.
The following analysis establishes minimum-norm
strong convergence and fast rates under this criterion.

\Needspace{5\baselineskip}
To analyze the discrete energy sequences, we first give a discrete
counterpart of Lemma~\ref{lem:comparison-continuous}.

\begin{lemma}\label{lem:comparison-discrete}
	Let $\lambda>0$, and let $\{H_k\}_{k\geq0}$ be a real sequence.
	Assume that there exist an integer $k_0\geq1$ and real
	sequences $\{e_k\}_{k\geq1}$ and $\{b_k\}_{k\geq1}$ satisfying
	\[
	e_k\to0,\qquad b_k\geq0,\qquad
	\sum_{k=1}^{\infty}b_k<+\infty,
	\]
	such that
	\[
	\left(1+\frac{\lambda}{k}\right)H_k
	\leq H_{k-1}+\frac{e_k}{k}+b_k
	\qquad\text{for all }k\geq k_0.
	\]
	Then $\limsup_{k\to\infty}H_k\leq0$.
\end{lemma}
\begin{proof}
	Fix $\varepsilon>0$. Since $e_k\to0$, there exists	an integer $m\geq k_0$ such that $e_k\leq\varepsilon$
	for all $k\geq m$. Hence,
	\[
	\left(1+\frac{\lambda}{k}\right)
	\left(H_k-\frac{\varepsilon}{\lambda}\right)
	\leq H_{k-1}-\frac{\varepsilon}{\lambda}+b_k.
	\]
	Iterating this inequality gives
	\[
	\begin{aligned}
		H_k-\frac{\varepsilon}{\lambda}
		&\leq
		\left(H_{m-1}-\frac{\varepsilon}{\lambda}\right)
		\prod_{j=m}^{k}
		\left(1+\frac{\lambda}{j}\right)^{-1}+
		\sum_{i=m}^{k}b_i
		\prod_{j=i}^{k}
		\left(1+\frac{\lambda}{j}\right)^{-1}.
	\end{aligned}
	\]
	For all sufficiently large $j$,	$\log(1+\lambda/j)\geq\lambda/(2j)$. Since the harmonic series diverges, for every fixed $i\geq m$,
	\[
	\prod_{j=i}^{k}
	\left(1+\frac{\lambda}{j}\right)^{-1}
	\longrightarrow0
	\qquad\text{as }k\to\infty.
	\]
	Moreover, for any fixed integer $N\geq m$ and $k>N$,
	\[
	\begin{aligned}
		0\leq
		\sum_{i=m}^{k}b_i
		\prod_{j=i}^{k}
		\left(1+\frac{\lambda}{j}\right)^{-1}
	\leq
		\sum_{i=m}^{N}b_i
		\prod_{j=i}^{k}
		\left(1+\frac{\lambda}{j}\right)^{-1}
		+\sum_{i=N+1}^{\infty}b_i.
	\end{aligned}
	\]
	Letting first $k\to\infty$ and then $N\to\infty$,	the summability of $\{b_k\}$ implies
	\[
	\sum_{i=m}^{k}b_i
	\prod_{j=i}^{k}
	\left(1+\frac{\lambda}{j}\right)^{-1}
	\longrightarrow0.
	\]
	Consequently,
	$\limsup_{k\to\infty}H_k\leq\varepsilon/\lambda$.
	Letting $\varepsilon\downarrow0$ yields the desired conclusion.
	This completes the proof.
\end{proof}

\Needspace{5\baselineskip}
Using this lemma, we establish minimum-norm strong convergence
and fast rates for Algorithm~\ref{alg:relative} under the relative residual criterion.

\begin{theorem}\label{thm:discrete}
	Under Assumption~\ref{ass:main} and \eqref{eq:parameters},
	let $\{(x_k,y_k)\}$ be generated by
	Algorithm~\ref{alg:relative} with $s>0$ and
	$0<2\sigma<r-2$. Then
	\begin{equation}\label{eq:discrete-strong}
		x_k\longrightarrow x^*,
		\qquad
		y_k\longrightarrow y^*
		\quad\text{strongly as }k\to+\infty.
	\end{equation}
	Moreover,
	\begin{equation}\label{eq:discrete-rates}
		\begin{aligned}
			&\|x_k-x_{k-1}\|+\|y_k-y_{k-1}\|
			=o(k^{-1}),
			\qquad
			\mathcal R(x_k,y_k)=O(k^{-2}),\\
			&\mathcal L(x_k,y^*)-\mathcal L(x^*,y_k)
			=o(k^{-2}),
		\end{aligned}
	\end{equation}
	where
	\[
	\mathcal R(x_k,y_k)
	:=
	\left(
	\|\nabla f(x_k)+\nabla_x\Psi(x_k,y_k)\|^2
	+
	\|\nabla g(y_k)-\nabla_y\Psi(x_k,y_k)\|^2
	\right)^{1/2}
	\]
	denotes the stationarity residual norm.
\end{theorem}

\begin{proof}
	Let $z^*=(x^*,y^*)$, and write
	$\Delta z_k=z_k-z_{k-1}$ and
	$\Delta q_k=q_k-q_{k-1}$.
	Throughout the proof, $C>0$ denotes a constant
	independent of $k$, whose value may change
	from one estimate to another.
	
	By the convexity assumptions, for any
	$z=(x,y)$ and $\widehat z=(\widehat x,\widehat y)$,
	\[
	\mathcal L(x,\widehat y)-\mathcal L(\widehat x,y)
	\le \langle F(z),z-\widehat z\rangle.
	\]
	Interchanging $z$ and $\widehat z$ and adding
	the resulting inequalities shows that $F$ is
	monotone. In particular, since $F(z^*)=0$,
	\[
	\langle F(z_k),z_k-z^*\rangle\ge0.
	\]
	
	We first establish boundedness by using the discrete counterpart
	of the continuous energy. Let $\varpi$ and $\mu>0$ be as in
	\eqref{eq:weight} and \eqref{eq:coercivity}, respectively.
	Define
	\begin{equation}\label{eq:discrete-energy}
		\mathcal E_k
		=\frac12\|z_k-z^*\|^2
		+\frac{\varpi}{2}
		\left\|q_k-\frac da z^*\right\|^2.
	\end{equation}
	Using $ar-d=c$, we rewrite
	\eqref{eq:discrete-compact} as
	\begin{align*}
		\Delta z_k
		&=-\theta s^2kF(z_k)
		-\frac{r(z_k-z^*)-(q_k-(d/a)z^*)}{k}
		-\frac{c}{ak}z^*+\xi_k,\\
		\Delta q_k
		&=-\frac{a(q_k-(d/a)z^*)-d(z_k-z^*)}{k}
		-\frac dk\Delta z_k.
	\end{align*}
	Taking inner products with $z_k-z^*$ and
	$\varpi(q_k-(d/a)z^*)$, respectively, and using
	\[
	2\langle u-v,u\rangle
	=\|u\|^2-\|v\|^2+\|u-v\|^2,
	\]
	we obtain
	\begin{align}\label{eq:discrete-energy-identity}
		&\mathcal E_k-\mathcal E_{k-1}
		+\frac12\|\Delta z_k\|^2
		+\frac{\varpi}{2}\|\Delta q_k\|^2
		+\theta s^2k\langle F(z_k),z_k-z^*\rangle
		\notag\\
		&\quad+\frac1k\left\{
		r\|z_k-z^*\|^2
		+\varpi a\left\|q_k-\frac da z^*\right\|^2
		-(1+\varpi d)
		\left\langle z_k-z^*,q_k-\frac da z^*\right\rangle
		\right\}
		\notag\\
		&=-\frac{c}{ak}\langle z^*,z_k-z^*\rangle
		+\langle\xi_k,z_k-z^*\rangle
		-\frac{\varpi d}{k}
		\left\langle\Delta z_k,q_k-\frac da z^*\right\rangle.
	\end{align}
	By \eqref{eq:coercivity}, the expression in braces
	is bounded below by $2\mu\mathcal E_k$.
	Moreover, the error estimate in
	\eqref{eq:discrete-compact} gives
	\begin{align*}
		|\langle\xi_k,z_k-z^*\rangle|
		+\frac{\varpi|d|}{k}
		\left|\left\langle
		\Delta z_k,q_k-\frac da z^*
		\right\rangle\right|
		\!\le\!\frac{\|\Delta z_k\|}{k}
		\left(
		\sigma\|z_k-z^*\|
		+\varpi|d|
		\left\|q_k-\frac da z^*\right\|
		\right)
		\!\le\!\frac14\|\Delta z_k\|^2
		+\frac{C}{k^2}\mathcal E_k.
	\end{align*}
	Since
	\[
	\frac ca\|z^*\|\|z_k-z^*\|
	\le\mu\mathcal E_k
	+\frac{c^2\|z^*\|^2}{2\mu a^2},
	\]
	it follows that, for all sufficiently large $k$,
	\[
	\left(1+\frac{\mu}{2k}\right)\mathcal E_k
	\le\mathcal E_{k-1}+\frac Ck.
	\]
	An induction argument shows that $\{\mathcal E_k\}$
	is bounded. Hence $\{z_k\}$ and $\{q_k\}$ are bounded,
	and the auxiliary equation in
	\eqref{eq:discrete-compact} yields
	$\|\Delta q_k\|\le C/k$.
	
	To obtain a preliminary estimate for the increments, subtract
	$k/(k-1)$ times the preceding position equation from the current one.
	This gives,
	for $k\ge3$,
	\begin{align}\label{eq:increment-identity}
		&\left(1+\frac rk\right)\Delta z_k
		+\theta s^2k\bigl(F(z_k)-F(z_{k-1})\bigr)
		\notag\\
		&=\frac{k}{k-1}\Delta z_{k-1}
		+\frac{\Delta q_k}{k}
		+\left(\frac{k}{(k-1)^2}-\frac1k\right)
		(rz_{k-1}-q_{k-1})
		+\xi_k-\frac{k}{k-1}\xi_{k-1}.
	\end{align}
	The two terms involving the auxiliary and position
	variables on the right are $O(k^{-2})$.
	Taking the inner product with $\Delta z_k$,
	and using monotonicity and the error bound, gives
	\[
	\left(1+\frac{r-\sigma}{k}\right)
	\|\Delta z_k\|
	\le
	\left(\frac{k}{k-1}
	+\frac{\sigma k}{(k-1)^2}\right)
	\|\Delta z_{k-1}\|+\frac{C}{k^2}.
	\]
	Thus, setting $W_k=k\|\Delta z_k\|$, we have
	\[
	W_k\le\rho_kW_{k-1}+\frac Ck,
	\qquad
	\rho_k=
	\frac{k^3(k-1+\sigma)}
	{(k-1)^3(k+r-\sigma)}.
	\]
	Since
	\[
	\lim_{k\to+\infty}k(1-\rho_k)
	=r-2-2\sigma>0,
	\]
	there exists $\nu>0$ such that, for all sufficiently
	large $k$,
	\[
	W_k\le\left(1-\frac{\nu}{k}\right)W_{k-1}
	+\frac Ck.
	\]
	Another induction argument shows that $\{W_k\}$
	is bounded. Consequently,
	\[
	\|\Delta z_k\|=O(k^{-1}),
	\qquad
	\|\xi_k\|=O(k^{-2}).
	\]
	Rearranging the position equation yields
	\begin{equation}\label{eq:scaled-residual-identity}
		\theta s^2k^2F(z_k)
		=-k\Delta z_k-rz_k+q_k+k\xi_k,
	\end{equation}
	and therefore $\|F(z_k)\|=O(k^{-2})$.
	
	Having established boundedness and residual decay, we now identify
	the weak cluster points of $\{z_k\}$ in order to prove strong
	convergence. Let $z_{k_j}\rightharpoonup\bar z$.
	For every $u\in\mathcal H$, monotonicity gives
	\[
	\langle F(u),u-z_{k_j}\rangle
	\ge\langle F(z_{k_j}),u-z_{k_j}\rangle.
	\]
	Since $F(z_{k_j})\to0$, passing to the limit yields
	$\langle F(u),u-\bar z\rangle\ge0$.
	Taking $u=\bar z-\tau F(\bar z)$, dividing by
	$\tau>0$, and letting $\tau\downarrow0$, we obtain
	$-\|F(\bar z)\|^2\ge0$ by continuity.
	Thus $F(\bar z)=0$, and hence $\bar z\in\mathcal S$.
	
	The characterization of the minimum-norm saddle point
	gives $\langle z^*,u-z^*\rangle\ge0$ for every
	$u\in\mathcal S$. Since $\{z_k\}$ is bounded and
	all its weak cluster points belong to $\mathcal S$,
	a subsequence argument yields
	\[
	\liminf_{k\to+\infty}
	\langle z^*,z_k-z^*\rangle\ge0.
	\]
	Returning to \eqref{eq:discrete-energy-identity},
	we consequently obtain
	\[
	\left(1+\frac{2\mu}{k}\right)\mathcal E_k
	\le\mathcal E_{k-1}+\frac{e_k}{k},
	\]
	where
	\begin{align*}
		e_k={}\frac ca
		\bigl(-\langle z^*,z_k-z^*\rangle\bigr)_+
		+k\|\xi_k\|\|z_k-z^*\|
		+\varpi|d|\|\Delta z_k\|
		\left\|q_k-\frac da z^*\right\|
		\longrightarrow0.
	\end{align*}
	Here $(u)_+=\max\{u,0\}$.
	Lemma~\ref{lem:comparison-discrete} implies
	$\mathcal E_k\to0$. Therefore,
	\[
	z_k\longrightarrow z^*,
	\qquad
	q_k\longrightarrow\frac da z^*,
	\]
	which proves \eqref{eq:discrete-strong}.
	
	To improve the estimate for the increments, set
	$w_k=k\Delta z_k$ and define
	\begin{equation}\label{eq:second-energy}
		\mathcal V_k
		=\frac12\|w_k\|^2
		-\frac ca\bigl(\|z_k\|^2-\|z^*\|^2\bigr).
	\end{equation}
	Multiplying \eqref{eq:increment-identity} by
	$(k-1)^2/k$, and substituting
	\[
	z_{k-1}=z_k-\frac{w_k}{k},
	\qquad
	q_{k-1}
	=q_k+\frac{aq_k-dz_k}{k}+\frac{d}{k^2}w_k,
	\]
	we obtain, using $ar-d=c$,
	\begin{align}\label{eq:weighted-velocity}
		&w_k-w_{k-1}
		+\left(\frac{r-2}{k}+\frac{1+d}{k^2}\right)w_k
		+\theta s^2(k-1)^2
		\bigl(F(z_k)-F(z_{k-1})\bigr)
		\notag\\
		&=\frac ca\left(\frac2k-\frac1{k^2}\right)z_k
		-\left(\frac{a+2}{k}-\frac1{k^2}\right)
		\left(q_k-\frac da z_k\right)
		+\frac{(k-1)^2}{k}
		\left(\xi_k-\frac{k}{k-1}\xi_{k-1}\right).
	\end{align}
	Taking the inner product with $w_k$ and using
	\[
	2\langle z_k,w_k\rangle
	=k\bigl(
	\|z_k\|^2-\|z_{k-1}\|^2+\|\Delta z_k\|^2
	\bigr),
	\]
	together with the polarization identity, gives
	\begin{align}\label{eq:second-energy-full}
		&\mathcal V_k-\mathcal V_{k-1}
		+\frac{r-2}{k}\|w_k\|^2
		+\theta s^2(k-1)^2
		\langle F(z_k)-F(z_{k-1}),w_k\rangle
		\notag\\
		&\le-\frac{1+d}{k^2}\|w_k\|^2
		-\frac{c}{2ak}
		\bigl(\|z_k\|^2-\|z_{k-1}\|^2\bigr)
		+\frac ca\left(1-\frac1{2k}\right)\|\Delta z_k\|^2
		\notag\\
		&\quad-\left(a+2-\frac1k\right)
		\left\langle q_k-\frac da z_k,\Delta z_k\right\rangle
		+\frac{(k-1)^2}{k}
		\left\langle
		\xi_k-\frac{k}{k-1}\xi_{k-1},w_k
		\right\rangle.
	\end{align}
	The gradient term on the left is nonnegative.
	Since $\{z_k\}$ and $\{w_k\}$ are bounded, the
	first four terms on the right are bounded above
	by $\delta_k/k$, where
	\[
	\delta_k=\frac Ck
	+C\left\|q_k-\frac da z_k\right\|
	\longrightarrow0.
	\]
	Moreover, the relative error bound implies
	\begin{align}\label{eq:relative-error-pair}
		&\frac{(k-1)^2}{k}
		\left\langle
		\xi_k-\frac{k}{k-1}\xi_{k-1},w_k
		\right\rangle
		\le\frac{\sigma}{k}\|w_k\|^2
		+\frac{\sigma}{k-1}\|w_{k-1}\|\|w_k\|
		\notag\\
		&\le\frac{\sigma}{k}\|w_k\|^2
		+\frac{\sigma}{2(k-1)}
		\bigl(\|w_k\|^2+\|w_{k-1}\|^2\bigr).
	\end{align}
	Hence
	\[
	\mathcal V_k-\mathcal V_{k-1}
	+A_k\|w_k\|^2
	\le B_k\|w_{k-1}\|^2+\frac{\delta_k}{k},
	\]
	where
	\[
	A_k=\frac{r-2-\sigma}{k}
	-\frac{\sigma}{2(k-1)},
	\qquad
	B_k=\frac{\sigma}{2(k-1)}.
	\]
	Let
	\[
	h_k=\frac ca\bigl(\|z_k\|^2-\|z^*\|^2\bigr).
	\]
	Then $h_k\to0$ and
	$\|w_k\|^2=2\mathcal V_k+2h_k$, so
	\[
	(1+2A_k)\mathcal V_k
	\le(1+2B_k)\mathcal V_{k-1}
	+\frac{\widetilde e_k}{k},
	\]
	with
	\[
	\widetilde e_k
	=\delta_k+2k|A_k||h_k|
	+2kB_k|h_{k-1}|
	\longrightarrow0.
	\]
	Dividing by $1+2B_k$, we obtain
	\[
	\left(1+\frac{\lambda_k}{k}\right)\mathcal V_k
	\le\mathcal V_{k-1}
	+\frac{\widetilde e_k}{k(1+2B_k)},
	\qquad
	\lambda_k=\frac{2k(A_k-B_k)}{1+2B_k}.
	\]
	Since
	\[
	\lambda_k\longrightarrow
	\lambda:=2(r-2-2\sigma)>0
	\]
	and $\{\mathcal V_k\}$ is bounded, it follows that
	\[
	\left(1+\frac{\lambda}{k}\right)\mathcal V_k
	\le\mathcal V_{k-1}+\frac{\varepsilon_k}{k},
	\qquad
	\varepsilon_k
	=\frac{\widetilde e_k}{1+2B_k}
	+|\lambda-\lambda_k|\,|\mathcal V_k|
	\longrightarrow0.
	\]
	Lemma~\ref{lem:comparison-discrete} therefore gives
	$\limsup_{k\to+\infty}\mathcal V_k\le0$.
	On the other hand, $\mathcal V_k\ge-h_k$ implies
	$\liminf_{k\to+\infty}\mathcal V_k\ge0$.
	Consequently, $\mathcal V_k\to0$ and $w_k\to0$,
	which proves
	$\|\Delta z_k\|=o(k^{-1})$.
	
	Finally, the saddle inequalities and convexity give
	\[
	0\le\mathcal L(x_k,y^*)-\mathcal L(x^*,y_k)
	\le\langle F(z_k),z_k-z^*\rangle
	\le\frac{C}{k^2}\|z_k-z^*\|.
	\]
	The strong convergence of $\{z_k\}$ yields the
	primal--dual gap estimate in
	\eqref{eq:discrete-rates}.
	The remaining estimates follow from
	$\mathcal R(x_k,y_k)=\|F(z_k)\|=O(k^{-2})$
	and
	\[
	\|\Delta x_k\|+\|\Delta y_k\|
	\le\sqrt{2}\|\Delta z_k\|=o(k^{-1}).
	\]
	This completes the proof.
\end{proof}

\Needspace{5\baselineskip}
As in the continuous case, the preceding convergence result yields
precise limits for the scaled stationarity residuals.

\begin{corollary}\label{cor:scaled-discrete}
	Under the assumptions of Theorem~\ref{thm:discrete},
	as $k\to+\infty$,
	\begin{equation}\label{eq:scaled-discrete-limits}
		\begin{aligned}
			s^2k^2\bigl(
			\nabla f(x_k)+\nabla_x\Psi(x_k,y_k)
			\bigr)
			&\longrightarrow
			-\frac{c}{1-2\theta}x^*,\\
			s^2k^2\bigl(
			\nabla g(y_k)-\nabla_y\Psi(x_k,y_k)
			\bigr)
			&\longrightarrow
			-\frac{c}{1-2\theta}y^*.
		\end{aligned}
	\end{equation}
	Consequently,
	\[
	\mathcal R(x_k,y_k)=
	\begin{cases}
		\Theta(k^{-2}),
		& (x^*,y^*)\ne(0,0),\\
		o(k^{-2}),
		& (x^*,y^*)=(0,0),
	\end{cases}
	\]
	where $\mathcal R$ is the stationarity residual
	norm defined in Theorem~\ref{thm:discrete}.
\end{corollary}

\begin{proof}
	Let $z_k=(x_k,y_k)$ and $z^*=(x^*,y^*)$.
	Theorem~\ref{thm:discrete} and its proof give
	\[
	z_k\longrightarrow z^*,
	\qquad
	q_k\longrightarrow\frac da z^*,
	\qquad
	k(z_k-z_{k-1})\longrightarrow0.
	\]
	Moreover, the relative error bound in
	\eqref{eq:discrete-compact} implies
	\[
	k\|\xi_k\|
	\le\sigma\|z_k-z_{k-1}\|
	\longrightarrow0.
	\]
	Passing to the limit in
	\eqref{eq:scaled-residual-identity}, and using
	$d=ar-c$ and $a\theta=1-2\theta$, yields
	\[
	s^2k^2F(z_k)
	\longrightarrow
	\frac1\theta\left(-r+\frac da\right)z^*
	=-\frac{c}{1-2\theta}z^*.
	\]
	This proves \eqref{eq:scaled-discrete-limits}.
	Taking norms, we obtain
	\[
	\lim_{k\to+\infty}k^2\mathcal R(x_k,y_k)
	=\frac{c}{s^2(1-2\theta)}
	\sqrt{\|x^*\|^2+\|y^*\|^2}.
	\]
	The asserted alternatives follow according
	to whether this limit is positive or zero.
	This completes the proof.
\end{proof}

\Needspace{5\baselineskip}
The gap estimate also gives rates for the function values and gradients
with one variable fixed at the minimum-norm saddle point.

\begin{corollary}\label{cor:fixed-discrete}
	Under the assumptions of Theorem~\ref{thm:discrete},
	as $k\to+\infty$,
	\begin{equation}\label{eq:discrete-slice-values}
		\begin{aligned}
			f(x_k)-f(x^*)
			+\Psi(x_k,y^*)-\Psi(x^*,y^*)
			&=o(k^{-2}),\\
			g(y_k)-g(y^*)
			+\Psi(x^*,y^*)-\Psi(x^*,y_k)
			&=o(k^{-2}).
		\end{aligned}
	\end{equation}
	Moreover,
	\begin{equation}\label{eq:discrete-slice-gradients}
		\|\nabla f(x_k)+\nabla_x\Psi(x_k,y^*)\|
		+\|\nabla g(y_k)-\nabla_y\Psi(x^*,y_k)\|
		=o(k^{-1}).
	\end{equation}
\end{corollary}

\begin{proof}
	By the saddle inequalities, both expressions
	in \eqref{eq:discrete-slice-values} are
	nonnegative, and their sum equals
	$\mathcal L(x_k,y^*)-\mathcal L(x^*,y_k)$.
	Thus \eqref{eq:discrete-slice-values} follows
	from \eqref{eq:discrete-rates}.
	
	Moreover, Assumption~\ref{ass:main} implies
	that the gradients of
	$x\mapsto\mathcal L(x,y^*)$ and
	$y\mapsto-\mathcal L(x^*,y)$ are locally
	Lipschitz near $x^*$ and $y^*$, respectively.
	Since $(x_k,y_k)\to(x^*,y^*)$ strongly,
	the local gradient estimate established in
	the proof of Corollary~\ref{cor:continuous-slices}
	gives, for some $L_*>0$ and all sufficiently
	large $k$,
	\begin{align*}
		\|\nabla f(x_k)+\nabla_x\Psi(x_k,y^*)\|^2
		&\le 2L_*
		\bigl(\mathcal L(x_k,y^*)-\mathcal L(x^*,y^*)\bigr),\\
		\|\nabla g(y_k)-\nabla_y\Psi(x^*,y_k)\|^2
		&\le 2L_*
		\bigl(\mathcal L(x^*,y^*)-\mathcal L(x^*,y_k)\bigr).
	\end{align*}
	Consequently,
	\begin{align*}
		&\bigl(
		\|\nabla f(x_k)+\nabla_x\Psi(x_k,y^*)\|
		+\|\nabla g(y_k)-\nabla_y\Psi(x^*,y_k)\|
		\bigr)^2\\
		&\le 4L_*
		\bigl(\mathcal L(x_k,y^*)-\mathcal L(x^*,y_k)\bigr)
		=o(k^{-2}),
	\end{align*}
	which proves \eqref{eq:discrete-slice-gradients}.
	This completes the proof.
\end{proof}

\begin{remark}\label{rem:comparison-literature}
The preceding results can be compared with related discrete methods.
The implicit Fast OGDA method \cite{BotCsetnekNguyen2025} gives
$o(k^{-2})$ residuals for suitable scaling, with weak convergence
but no minimum-norm selection guarantee. Anchored extragradient
methods \cite{BotChenchene2026} provide this selection property.
In this respect, Theorem~\ref{thm:discrete} combines minimum-norm
selection with an $O(k^{-2})$
stationarity residual, an $o(k^{-2})$ gap at the fixed saddle point,
and $o(k^{-1})$ increments under local Lipschitz continuity.
These conclusions hold under a computable relative residual criterion.
The numerical experiments also report the work required by the inner solvers.
\end{remark}

\section{Numerical Experiments}\label{Numerical}\label{sec:numerics}
\begingroup
\setlength{\intextsep}{10pt plus 2pt minus 2pt}
\setlength{\textfloatsep}{12pt plus 2pt minus 2pt}
\setlength{\floatsep}{10pt plus 2pt minus 2pt}

In this section, we illustrate the convergence results and compare
the proposed methods with related methods. The experiments are
implemented in MATLAB R2021b and consist of two nonlinear saddle
problems and one bilinear problem with 32, 2,000 and 20,000 variables,
respectively. In these structured examples, the active variables
determine stationarity, while the flat nullspace variables allow us
to examine minimum-norm selection. Accordingly, for $z=(x,y)$, we report
\begin{equation}\label{eq:metrics}
 \mathcal R(z)=\|F(z)\|,\qquad
 \mathcal G(z)=\mathcal L(x,y^*)-\mathcal L(x^*,y),\qquad
 D_*(z)=\|z-z^*\|,
\end{equation}
where $F(z)=(\nabla f(x)+\nabla_x\Psi(x,y),\nabla g(y)-\nabla_y\Psi(x,y))$
and $z^*=(x^*,y^*)$ is known explicitly.
To compare computational work as well as iteration counts, we record
all evaluations of $F$ made by the inner solvers and rejected trials,
excluding evaluations used solely for diagnostics.

All continuous comparisons use the same initial position and zero initial
velocity at $t_0=1$; the horizontal axis $t$ denotes model time.
The proposed algorithm is initialized by
$q_1=rz_1+\theta s^2F(z_1)$ and uses $\sigma=0.4$.
For nonlinear subproblems, the proposed algorithm and, where used,
implicit Fast OGDA employ Tseng's forward--backward--forward method
\cite{Tseng SICON 2000}. The proposed algorithm accepts an inner point $z$
when
\begin{equation}\label{eq:stop-num}
 \|z-v_k+\eta_kF(z)\|
 \le \frac{\sigma}{k+r}\|z-z_{k-1}\|,
\end{equation}
whereas the implicit Fast OGDA update is computed to normalized residual
$2\times10^{-12}(1+\|v_k\|)$. The continuous trajectories are obtained
from the equivalent first-order equations using
\texttt{ode15s}, with relative and absolute tolerances
$(2\times10^{-10},2\times10^{-12})$ for the nonlinear problems and
$(2\times10^{-9},2\times10^{-11})$ for the bilinear problems.
For the first initialization of Examples 2 and 3, reducing both tolerances
by a factor of ten changes the final residual by at most $0.19\%$.

To describe the inner implementation and justify the stopping test,
for each nonlinear subproblem write $F=A_0+M(\cdot-z^*)$ and
$B_k=M(\cdot-z^*)+\eta_k^{-1}(\cdot-v_k)$, where $M$ is the
nonnegative diagonal mapping specified below. Starting from $w=z_{k-1}$,
the inner updates are
\[
 y=(I+\gamma B_k)^{-1}(w-\gamma A_0(w)),\qquad
 w^+=y-\gamma(A_0(y)-A_0(w)).
\]
The componentwise backward step and the condition
$\gamma\|A_0(y)-A_0(w)\|\le0.8\|y-w\|$, together with
$\eta_k^{-1}$-strong monotonicity of $G_k$, give
\[
 \|w^+-\widehat z_k\|^2
 \le\|w-\widehat z_k\|^2-0.36\|w-y\|^2
 -2\gamma\eta_k^{-1}\|y-\widehat z_k\|^2.
\]
Consequently, the inner iterates $w$ remain bounded.
For each fixed $k$, the capped trial steps also keep $y$ bounded
in Example~1. Since these experiments are finite-dimensional,
local Lipschitz continuity and the halving rule give $\inf\gamma>0$;
Example~2 uses a fixed positive step. Summing the preceding inequality
therefore yields $w-y\to0$ and $y\to\widehat z_k$.
As explained after Algorithm~\ref{alg:relative}, the initial check and
$w\to\widehat z_k$ ensure finite termination of \eqref{eq:stop-num}
in exact arithmetic with $\sigma>0$.

\FloatBarrier\par\medskip\Needspace{7\baselineskip}
\noindent\textbf{Example 1: A nonlinear quartic--sine saddle problem.}\par\smallskip
\label{subsec:quartic}
We first consider the following nonlinear convex--concave saddle problem:
\begin{equation}\label{eq:quartic}
 \min_{x\in\mathbb R^{2h}}\max_{y\in\mathbb R^{2h}}\quad
 \mathcal L(x,y)=\frac\mu2(\|u\|^2-\|v\|^2)
 +\frac\tau4\sum_{i=1}^{h}(u_i^4-v_i^4)
 +\rho\sum_{i=1}^{h}\sin(u_i)\sin(v_i),
\end{equation}
where $x=(u,x_\perp)$, $y=(v,y_\perp)$ and each block belongs to
$\mathbb R^h$. We set $h=8$, $\mu=10^{-3}$, $\tau=0.1$ and
$\rho=9\times10^{-4}$. Since $\mu>\rho$, the objective is convex--concave
with a locally but not globally Lipschitz gradient.
The saddle set is $\{(x,y):u=v=0\}$, with $z^*=0$;
$x_\perp,y_\perp$ do not enter the objective.

We compare the proposed methods with the continuous and implicit discrete
Fast OGDA methods of Bo\c{t} et al.
\cite{BotCsetnekNguyen2025}. The parameters are chosen as follows:
\begin{itemize}
\item Proposed methods: $\alpha=9$, $\theta=0.2$, $c=16$ and discrete
stepsize $s=4$.
\item Fast OGDA: $\alpha=9$ and $\beta(t)=0.2t$ for the dynamical system;
$\beta_j=j+1$ and stepsize $s^2$ with $s=16$ for the implicit algorithm.
\end{itemize}
Both methods use the stored initial vector $z_0=2\omega$,
$\omega\sim N(0,I_{32})$, with $t\in[1,300]$ and $k\le400$.
Both inner solvers use $M=0$. At each subproblem, initialize
$\gamma_{\rm prev}=1/(\mu+2\rho)$ and start each trial with
$\gamma=\min\{1.5\gamma_{\rm prev},10^8\}$.
Halve $\gamma$ until $\gamma\|F(y)-F(w)\|\le0.8\|y-w\|$,
and set $\gamma_{\rm prev}=\gamma$ after acceptance.
No Hessian is evaluated.

\begin{figure}[!ht]
\centering
\includegraphics[width=1\textwidth]{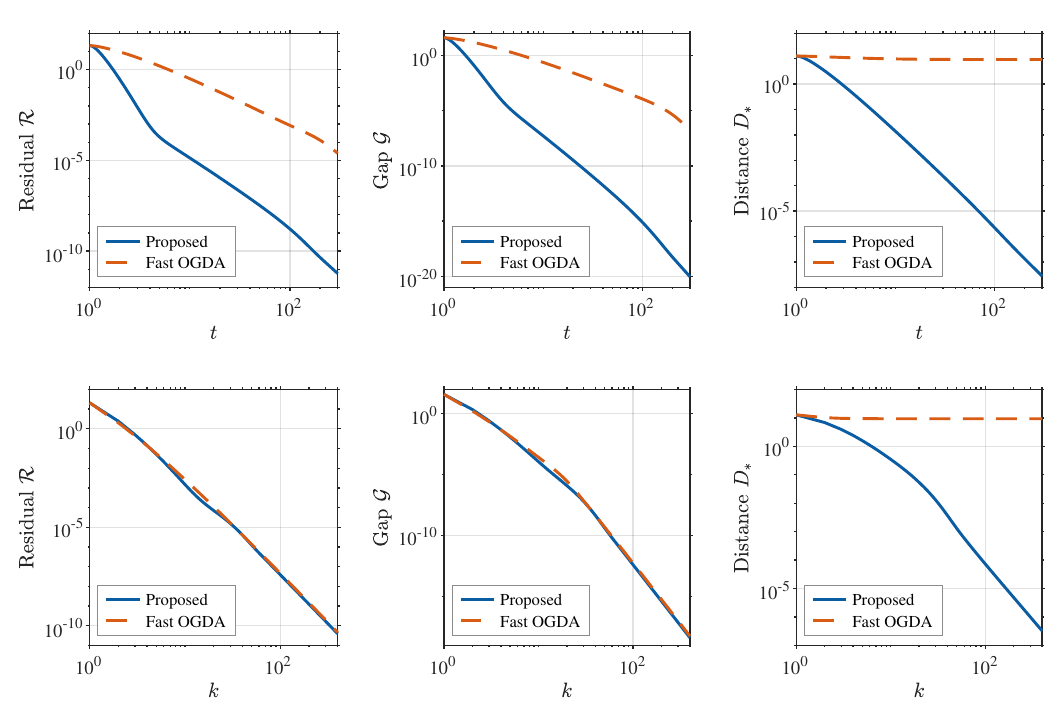}
\caption{Comparisons of the proposed methods and Fast OGDA for Example 1. The upper and lower panels show continuous and discrete results, respectively.}
\label{fig:quartic}
\end{figure}
\FloatBarrier

Figure~\ref{fig:quartic} shows faster continuous decay of all three
measures for the proposed dynamical system. Both discrete residuals
approach $10^{-11}$, whereas the final minimum-norm distances are
$3.26\times10^{-7}$ for the proposed algorithm and $9.30$ for Fast OGDA,
illustrating the selection effect.

\begin{figure}[!ht]
\centering
\includegraphics[width=1\textwidth]{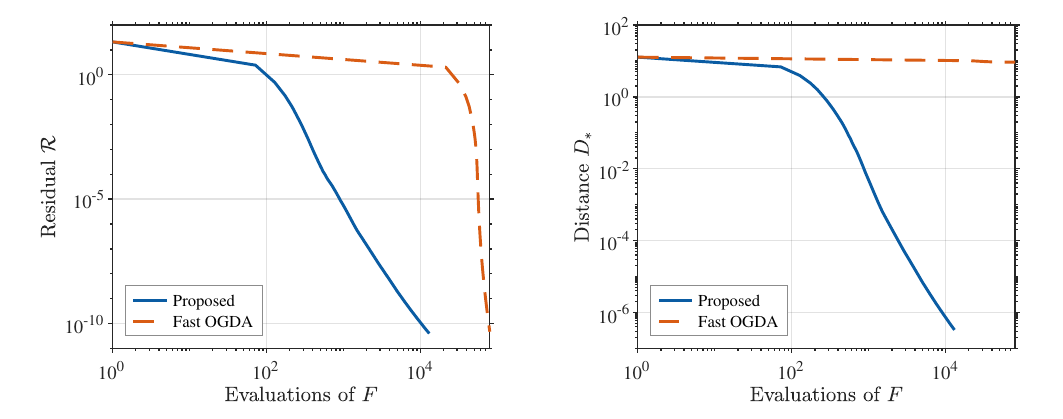}
\caption{Stationarity residual and minimum-norm distance against gradient evaluations for Example 1.}
\label{fig:quartic-work}
\end{figure}
\FloatBarrier

Figure~\ref{fig:quartic-work} further compares the residual and
minimum-norm distance against total evaluations, including all inner work.

\begin{table}[!ht]\centering\small
\caption{Numerical results of the proposed methods and Fast OGDA for Example 1.}
\label{tab:quartic}
\begin{tabular*}{\textwidth}{@{\extracolsep{\fill}}llrrrr@{}}\toprule
Model & Method & $\mathcal R$ & $\mathcal G$ & $D_*$ & $F$ calls\\\midrule
Continuous & Proposed & $5.81\times10^{-12}$ & $9.34\times10^{-21}$ & $2.77\times10^{-8}$ & --- \\
 & Fast OGDA & $2.46\times10^{-5}$ & $1.67\times10^{-7}$ & $9.3$ & --- \\
\midrule
Discrete & Proposed & $3.97\times10^{-11}$ & $4.35\times10^{-19}$ & $3.26\times10^{-7}$ & $13049$ \\
 & Fast OGDA & $4.66\times10^{-11}$ & $6.00\times10^{-19}$ & $9.3$ & $79808$ \\
\bottomrule
\end{tabular*}
\end{table}
\FloatBarrier

Table~\ref{tab:quartic} gives comparable final residuals with 13,049
and 79,808 evaluations, an $84\%$ reduction for these implementations.
To examine the contribution of the inner stopping tolerance,
we additionally impose the same high-accuracy absolute test on the
proposed updates. This raises the count to 25,160, with residual
$3.99\times10^{-11}$.

\FloatBarrier\par\medskip\Needspace{7\baselineskip}
\noindent\textbf{Example 2: A medium-scale nonlinear tanh saddle problem.}\par\smallskip
\label{subsec:tanh}
To examine a larger nonlinear problem and minimum-norm selection
when the selected saddle point is nonzero, we next consider
\begin{equation}\label{eq:tanh}
 \min_{x\in\mathbb R^{2h}}\max_{y\in\mathbb R^{2h}}\quad
 \mathcal L(x,y)=\frac12u^{\mathsf T}Du-\frac12v^{\mathsf T}Dv
 +\kappa u^{\mathsf T}v+\rho\sum_{i=1}^{h}\tanh(u_i)\tanh(v_i),
\end{equation}
where $u=x_{1:h}-p$ and $v=y_{1:h}-q$. We set $h=500$, so that $x,y\in\mathbb R^{1000}$, and choose
\begin{equation*}\begin{aligned}
 D_{ii}&=10^{-7(i-1)/(h-1)},\qquad \rho=5\times10^{-8},\qquad
 \kappa=2.5\times10^{-8},\\
 p_i&=10^{-3}\left(0.5+\frac{0.7(i-1)}{h-1}\right),\qquad
 q_i=-10^{-3}\left(0.4+\frac{0.6(i-1)}{h-1}\right).
\end{aligned}\end{equation*}
The objective is convex--concave because
$\lambda_{\min}(D)>4\rho/(3\sqrt3)$. Its gradient is globally Lipschitz,
with a valid bound $L=\|D\|+\kappa+2\rho$.
The minimum-norm saddle point is $z^*=((p,0),(q,0))$.
The nonuniform diagonal provides different curvature scales, with a
nonlinear coupling that is small relative to the largest diagonal entry.

In this experiment, we compare with explicit Fast OGDA
\cite{BotCsetnekNguyen2025} and G-EAG
\cite{BotChenchene2026}.
The parameters are chosen as follows:
\begin{itemize}
	\item Proposed methods: $\alpha=13$, $\theta=1/7$,
	$c=36$ and discrete stepsize $s=64$.
	\item Fast OGDA: $\alpha=13$ and $\beta(t)=t/7$
	for the dynamical system; $\alpha=3$ and stepsize
	$h_{\rm F}=0.48/L$ for the explicit algorithm.
	\item G-EAG: anchor zero, stepsize $h_{\rm G}=0.8/L$,
and $h_{\rm G}\varepsilon_j=1.5/(j+10)$, $j\ge0$.
\end{itemize}

To exploit the diagonal structure in the inner solves,
for the proposed algorithm we choose $M$ with entries
$m_i=\max\{D_{ii}-2\rho,0\}$ in each active block and zero entries
in the remaining blocks. The shift is given by the known data $p,q$.
We use
$\gamma=\min\{\sqrt{\eta_k/L_A},0.8/L_A\}$, where
$L_A=\max_i(D_{ii}-m_i)+\kappa+2\rho$, together with
\eqref{eq:stop-num}. Both discrete comparators are explicit and require no inner solver.
The diagonal backward step costs $O(h)$ operations. It reduces the
forward Lipschitz bound from $L\approx1$ to $L_A=2.25\times10^{-7}$. 
For the explicit comparison, the above G-EAG parameters belong
to the $A>1$ branch of \cite[Corollary 2.6]{BotChenchene2026}.

The methods use the same ten initial vectors $z_1=z^*+2\omega$,
$\omega\sim N(0,I_{2000})$, generated with seeds 61101--61110 and stored in the accompanying data files.
Explicit Fast OGDA is initialized with equal current, preceding and
auxiliary points. We integrate to $t=300$. For the discrete plots,
$k=1$ denotes the initial point and we display the first 2,000 updates
of every method. To measure computational work separately, we record
the first accepted iterate reaching each target within 20,000 evaluations
of $F$, including all inner evaluations for the proposed algorithm.

\begin{figure}[!ht]
\centering
\includegraphics[width=1\textwidth]{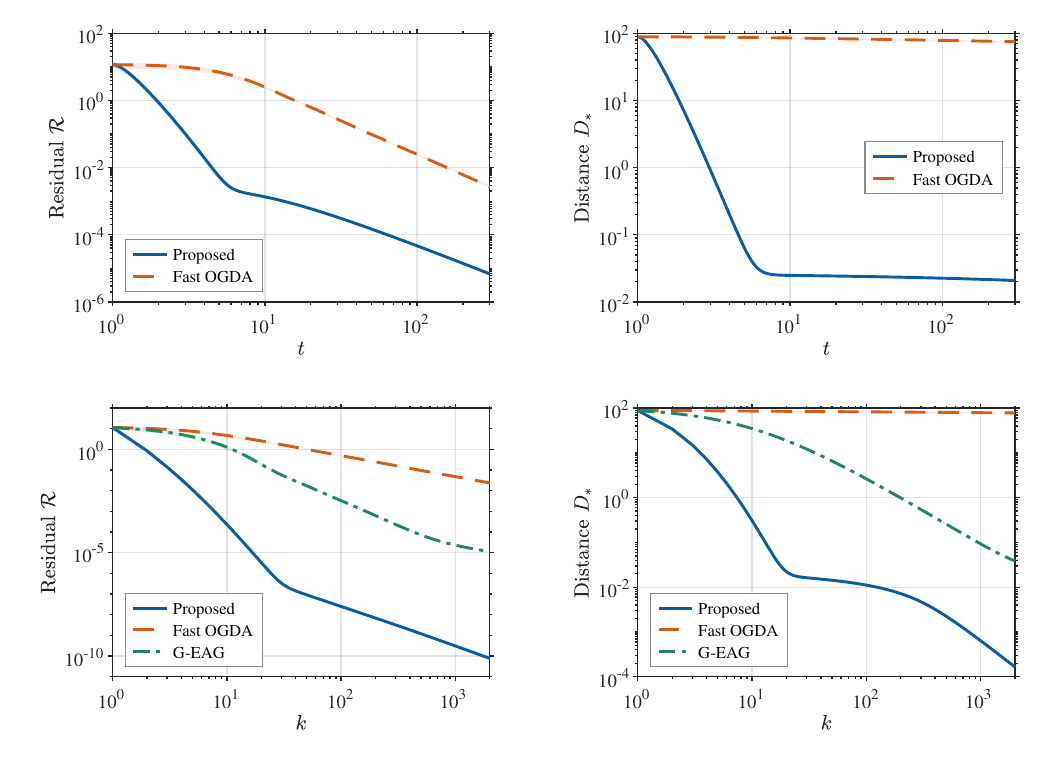}
\caption{Comparisons for Example 2. The upper panels show continuous trajectories against $t$; the lower panels show discrete iterates against $k$. Fast OGDA is explicit in the lower panels. Each method has one median curve, with shading showing the minimum--maximum range over ten initializations.}
\label{fig:tanh}
\end{figure}
\FloatBarrier

At $t=300$, Figure~\ref{fig:tanh} gives median residual and distance
$6.84\times10^{-6}$ and $2.08\times10^{-2}$ for the proposed dynamical
system, versus $2.69\times10^{-3}$ and $75.0$ for Fast OGDA.
At $k=2001$, the median residuals are $7.77\times10^{-11}$,
$2.41\times10^{-2}$ and $1.13\times10^{-5}$ for the proposed algorithm,
explicit Fast OGDA and G-EAG; their distances are $1.65\times10^{-4}$,
$78.1$ and $3.86\times10^{-2}$. The proposed methods give smaller
final residuals and distances in all ten runs.

\begin{figure}[!ht]
\centering
\includegraphics[width=\textwidth]{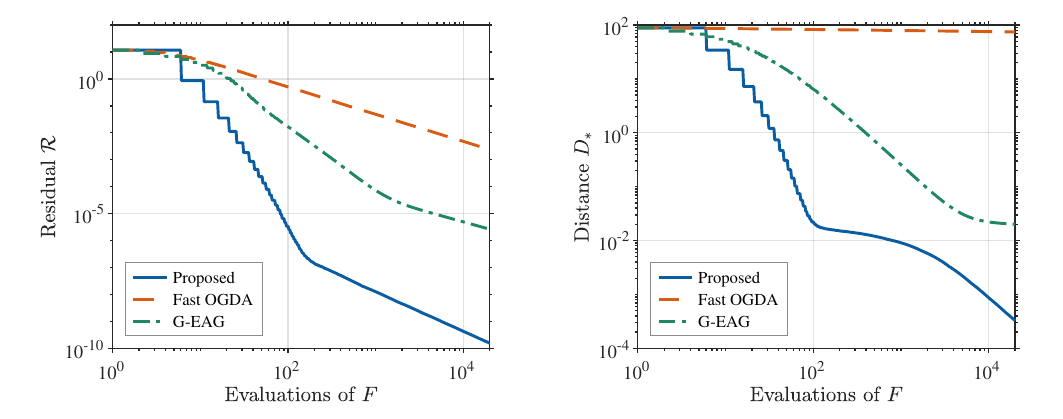}
\caption{Median stationarity residual and minimum-norm distance against cumulative evaluations of $F$ for Example 2. All inner evaluations and rejected trials are included.}
\label{fig:tanh-work}
\end{figure}
\FloatBarrier

To account for the different costs of an outer update,
Figure~\ref{fig:tanh-work} compares accepted iterates within 20,000
evaluations, holding the last accepted value between updates.
The proposed algorithm has smaller final residual and distance.
This work budget differs from the common iteration horizon in
Figure~\ref{fig:tanh}, where the proposed algorithm uses
32,067 evaluations to reach $k=2001$.

\begin{table}[!ht]\centering\small
\caption{Gradient evaluations required to reach the prescribed accuracy in Example 2.}
\label{tab:tanh}
\setlength{\tabcolsep}{4pt}
\begin{tabular*}{\textwidth}{@{\extracolsep{\fill}}lrrrrrr@{}}\toprule
 &\multicolumn{3}{c}{Stationarity residual $\mathcal R$}&\multicolumn{3}{c}{Minimum-norm distance $D_*$}\\
\cmidrule(lr){2-4}\cmidrule(lr){5-7}
Method &$10^{-6}$&$10^{-8}$&$10^{-10}$&$10^{-2}$&$10^{-3}$&$10^{-4}$\\\midrule
Proposed & 121 & 1\,173 & -- & 813 & 9\,100 & -- \\
Fast OGDA (explicit) & -- & -- & -- & -- & -- & -- \\
G-EAG & -- & -- & -- & -- & -- & -- \\
\bottomrule
\end{tabular*}
\footnotetext{
Entries are median cumulative evaluations of $F$, including all inner
subproblem evaluations. Every numerical entry represents ten successful
runs; ``--'' means no success within 20,000 evaluations.
}
\end{table}
\FloatBarrier

Table~\ref{tab:tanh} shows that the proposed algorithm reaches
$\mathcal R\le10^{-8}$ and $D_*\le10^{-3}$ with medians of
1,173 and 9,100 evaluations, respectively. Neither explicit Fast OGDA
nor G-EAG reaches either target within the prescribed budget.

\FloatBarrier\par\medskip\Needspace{7\baselineskip}
\noindent\textbf{Example 3: A large-scale sparse bilinear saddle problem.}\par\smallskip
\label{subsec:bilinear}
Finally, to examine the methods at a larger scale with non-diagonal
coupling, we consider the following sparse bilinear saddle problem
with $h=5000$:
\begin{equation}\label{eq:bilinear}
 \min_{x\in\mathbb R^{2h}}\max_{y\in\mathbb R^{2h}}\quad
 \frac12(Ux)^{\mathsf T}D(Ux)+(Ux)^{\mathsf T}K(Vy)
 -\frac12(Vy)^{\mathsf T}D(Vy),
\end{equation}
where $U,V\in\mathbb R^{h\times2h}$ have orthonormal rows.
Thus, the problem has 20,000 total variables. We choose the diagonal
matrix $D$ and symmetric tridiagonal matrix $K$ as follows:
\[
\begin{aligned}
 D_{11}&=1, & D_{ii}&=10^{-5+(i-2)/(h-2)}, &&i=2,\ldots,h,\\
 K_{11}&=0.1, & K_{ii}&=2.5\times10^{-6}, &&i=2,\ldots,h,\\
 &&K_{i,i+1}&=K_{i+1,i}=10^{-6}, &&i=1,\ldots,h-1.
\end{aligned}
\]
The matrix $K$ has 14,998 nonzero entries and couples neighboring
coordinates. Let $U=[C\ S]$ and $V=[S\ {-C}]$, where
$C=\operatorname{diag}(\cos\phi_i)$, $S=\operatorname{diag}(\sin\phi_i)$
and $\phi_i=0.15+1.2(i-1)/(h-1)$. The saddle set is
$\ker U\times\ker V$, and its minimum-norm element is $z^*=0$.
All methods are implemented in the same orthogonal active and nullspace
coordinates. The implicit subproblems are solved using sparse linear
algebra. Since their coefficients change with the outer index,
a new sparse factorization is used at each update.

We compare the proposed methods with Fast OGDA, G-EAG and the bilinear
primal--dual methods of He and Fang~\cite{He Fang 2026}, denoted by
He--Fang. The parameter choices are as follows:
\begin{itemize}
\item Proposed methods: $\alpha=9$, $\theta=0.2$, $c=16$ and $s=16$.
\item Fast OGDA: $\alpha=9$, $\beta(t)=0.2t$ for the dynamical system;
$\beta_j=j+1$ and stepsize $s^2$ with $s=16$ for the implicit algorithm.
\item He--Fang: $\alpha=5$, $\theta=1/3$ in continuous time;
$t_j=(j+3)/4$, $\eta=0.95$, and primal and dual stepsizes
$0.95/L_f$ and $0.95/L_g$ in discrete time.
\item G-EAG: anchor zero, $h_0=0.9/L$ with $L=1+\|K\|_\infty$, and
$h_0\varepsilon_j=1.5/(j+10)$.
\end{itemize}
We use three shared initial vectors in these coordinates: the all-ones
vector and two vectors $2\omega$, $\omega\sim N(0,I_{20000})$,
generated with seeds 73102 and 73103. The continuous interval is
$[1,300]$, and the discrete algorithms are run to $k=400$.
Performance is evaluated by the residual, minimum-norm distance and gap
$\mathcal G(z)=((Ux)^{\mathsf T}D(Ux)+(Vy)^{\mathsf T}D(Vy))/2$.

\begin{figure}[!ht]
\centering
\includegraphics[width=\textwidth]{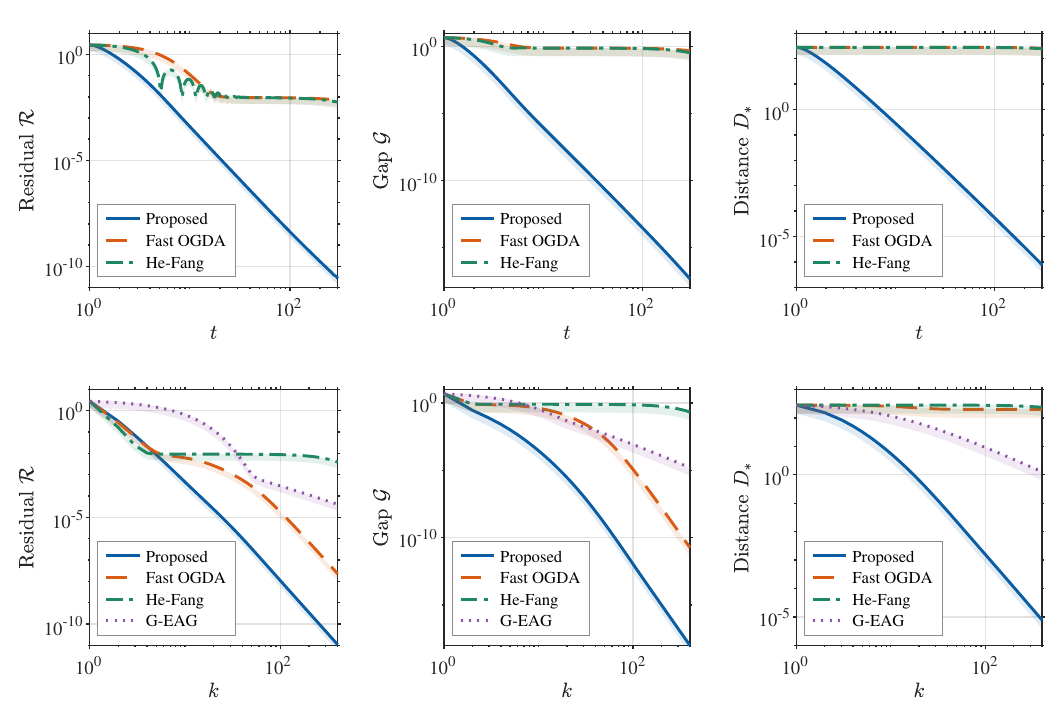}
\caption{Comparisons for Example 3 with 20,000 variables. The upper and lower panels show continuous and discrete results, respectively. Lines and shading show medians and minimum--maximum ranges over three initializations.}
\label{fig:bilinear}
\end{figure}
\FloatBarrier

Figure~\ref{fig:bilinear} shows that the proposed methods achieve smaller
final residuals, gaps and minimum-norm distances in all three runs.
These results demonstrate their numerical performance on a large sparse
problem with non-diagonal coupling.

\begin{table}[!ht]\centering\small
\caption{Median numerical results for Example 3 at $t=300$ and $k=400$.}
\label{tab:bilinear}
\begin{tabular*}{\textwidth}{@{\extracolsep{\fill}}llrrr@{}}\toprule
Model & Method & $\mathcal R$ & $\mathcal G$ & $D_*$\\\midrule
Continuous & Proposed & $2.58\times10^{-11}$ & $4.36\times10^{-18}$ & $7.68\times10^{-7}$ \\
 & Fast OGDA & $7.01\times10^{-3}$ & $4.96\times10^{-1}$ & $2.61\times10^{2}$ \\
 & He--Fang & $5.62\times10^{-3}$ & $3.53\times10^{-1}$ & $2.49\times10^{2}$ \\
\midrule
Discrete & Proposed & $1.10\times10^{-11}$ & $1.02\times10^{-18}$ & $7.77\times10^{-6}$ \\
 & Fast OGDA & $2.29\times10^{-8}$ & $1.83\times10^{-11}$ & $1.98\times10^{2}$ \\
 & He--Fang & $3.93\times10^{-3}$ & $2.09\times10^{-1}$ & $2.36\times10^{2}$ \\
 & G-EAG & $4.12\times10^{-5}$ & $1.56\times10^{-5}$ & $1.27$ \\
\bottomrule
\end{tabular*}
\end{table}
\FloatBarrier

Table~\ref{tab:bilinear} gives the corresponding endpoint values,
confirming the smallest residual and distance for the proposed methods
at the stated horizons. To complement these accuracy comparisons,
we report the computation times for the 399 outer updates, excluding
diagnostics. The median times are $3.90$, $4.02$, $0.67$ and $0.27$ seconds
for the proposed algorithm, implicit Fast OGDA, He--Fang and G-EAG,
respectively. The first three methods use 399 sparse solves each.

\FloatBarrier
\endgroup

\Needspace{9\baselineskip}
\section*{Declarations}

\noindent\textbf{Funding.}
This research was supported by the National Natural Science Foundation of China (12171070).

\medskip
\noindent\textbf{Conflict of interest.}
The authors declare that they have no conflict of interest.

\medskip
\noindent\textbf{Data availability.}
The MATLAB code and data supporting the numerical experiments
in this study are available from the corresponding author
upon reasonable request.

\end{document}